\documentclass[a4paper, oneside, 10pt]{amsart}
\usepackage{graphicx}
\usepackage{amsmath}
\usepackage{amssymb}
\usepackage{geometry}
\usepackage{tikz}
\usepackage[all]{xy}
\usepackage{amsthm}
\usepackage{amsfonts}
\usepackage[backref,bookmarks=true,colorlinks=true,pdfstartview=FitV,linkcolor=blue, citecolor=red,urlcolor=green]{hyperref}  
\usepackage{comment}
\newcommand{\Hit}{\operatorname{Hit}}
\newcommand{\PSL}{\operatorname{PSL}}

\newcommand{\bR}{\mathbb{R}}
\newcommand{\diag}{\operatorname{diag}}
\newcommand{\Tr}{\operatorname{Tr}}
\newcommand{\sign}{\operatorname{sign}}
\newcommand{\Ad}{\operatorname{Ad}}

\newcommand{\ad}{\operatorname{ad}}
\newcommand{\bS}{\mathbb{S}}
\newcommand{\kf}{\mathbf{\mathsf{B}}}
\newcommand{\ri}{\operatorname{i}}
\newcommand{\rd}{\operatorname{d}}
\newcommand{\re}{\operatorname{e}}
\renewcommand{\epsilon}{\varepsilon}

\newtheorem{lemma}{Lemma}[section]
\newtheorem{theorem}[lemma]{Theorem}
\newtheorem{corollary}[lemma]{Corollary}

\theoremstyle{definition}

\newtheorem{question}[lemma]{Question}

\newtheorem{example}[lemma]{Example}
\theoremstyle{remark}
\newtheorem{claim}[lemma]{Claim}

\title{Generic Properties of Hitchin Representations}
\author{Hongtaek Jung}
\address{Department of Mathematical Sciences\\ Seoul National University\\ Seoul} 
\begin{document}

\begin{abstract}
Let $G$ be a split real form of a complex simple adjoint group whose Weyl group contains $-1$, let $\lambda$ be the Jordan projection of $G$, and let $S$ be a closed orientable surface of genus at least 2. For a $G$-Hitchin representation $\rho$, we define the set $J(\rho):=\{\lambda(\rho(x))\,|\,x\in\pi_1(S)\setminus \{1\}\}$. Choose any hyperplane $H$ in the maximal abelian subalgebra of the Lie algebra of $G$.  Our main result shows that, for a generic $G$-Hitchin representation $\rho$, we have $J(\rho)\cap H=\emptyset$. As an application, we prove that generic orbifold Hitchin representations are strongly dense. This extends the result of Long, Reid, and Wolff for the Hitchin representations of surface groups. Our theorem also shows that the split real forms of many simple adjoint Lie groups contain strongly dense orbifold fundamental groups, partially generalizing the work of Breuillard, Guralnick, and Larsen.
\end{abstract}
\maketitle

\section{Introduction}
Let $G$ be a split real form of a complex simple adjoint Lie group and let $S$ be a closed orientable surface of genus at least 2. The $G$-Hitchin component, introduced by Hitchin \cite{hitchin92}, has been intensively studied as one of the prime examples of higher Teichm\"uller spaces. More generally, one can think of the $G$-Hitchin component for an orientable orbifold $O$ of negative Euler characteristic. For instance, Thurston \cite{Thurston80} considered the Teichm\"uller space for orbifolds and Choi and Goldman \cite{cgorb} explored the orbifold  $\PSL_3(\bR)$-Hitchin components as the deformation space of convex projective geometry. More recently, Alessandrini, Lee and Schaffhauser \cite{alessandrini} defined the orbifold $G$-Hitchin components in full generality. In this note, we explore two properties that generic $G$-Hitchin representations of orbifold groups have. 

Let $\overline{\Hit}_G(O)$ be the set of $G$-Hitchin representations and let $\Hit_G(O) := \overline{\Hit}_G(O)/G$. When $G=\PSL_n(\bR)$, we simply write $\overline{\Hit}_n(O) := \overline{\Hit}_{\PSL_n(\bR)}(O)$ and $\Hit_n(O) := \overline{\Hit}_{n}(O)/\PSL_n(\bR)$. We also denote by $\overline{\Hit}_G(S)$ and $\Hit_G(S)$ the space of $G$-Hitchin representations and the $G$-Hitchin component for a surface $S$.

The first generic property is about eigenvalues. For an element $[\rho]$ in the Fuchsian locus of $\Hit_{2k-1}(S)$, we have $\lambda_k(\rho(x))=0$ for all $x\in \pi_1(S)$, where $\lambda_k$ is the $\log$ of the absolute value of the $k$th largest eigenvalue. On the other hand, as we move away from the Fuchsian locus, we would expect to encounter Hitchin representations $\rho$ breaking the condition that $\lambda_k(\rho(x))=0$ for all $x\in \pi_1(S)$. Therefore, a descent question to ask is the following:
\begin{question}\label{qu:middleeigen}
Is there a Hitchin representation $[\rho]\in\Hit_{2k-1}(S)$ such that $\lambda_k (\rho (x))\ne 0$ for all $x\in \pi_1(S)\setminus \{1\}$?     
\end{question}
We will study Question~\ref{qu:middleeigen} in the orbifold setting. For orbifold Hitchin components, Question~\ref{qu:middleeigen} is a subtle problem even for $k=2$; as observed in  \cite[Proposition~1.1]{long2}, there is an infinite order element, say $x$, in some \emph{orbifold group} $\Delta$ such that $\lambda_2(\rho(x))=0$ for all $\rho\in \overline{\Hit}_3(\Delta)$. Yet, our main result gives an affirmative partial answer to Question~\ref{qu:middleeigen}. 

We will actually show more a generally statement under a restricted setting. Recall that any element $[\rho]$ in the Fuchsian locus of $\Hit_n(O)$ satisfies the relation 
\begin{equation}\label{eq:relation}
\lambda_i (\rho(x))+\lambda_{n-i+1}(\rho(x)) = 0
\end{equation}
for all $x\in \pi_1^{\mathrm{orb}}(O)$. Presumably, this equation does not hold once we move away from the Fuchsian locus. However, there is a group theoretic obstruction to break (\ref{eq:relation}). We say that an element $x\in \pi_1^{\mathrm{orb}}(O)$ is \emph{regular} if the normalizer of $x$ is the infinite cyclic generated by $x$ itself. If $x$ is not regular, $x$ and $x^{-1}$ are conjugate and this forces $\lambda_i (\rho(x))+\lambda_{n-i+1}(\rho(x)) = 0$ hold for any Hitchin representations $\rho$. We prove that, restricted to ones that survive under the abelianization, not being regular is the only obstruction to which $\lambda_i (\rho(x))+\lambda_{n-i+1}(\rho(x)) = 0$ fails. 

\begin{theorem}\label{thm:genericeigenvalue}
    Let $O$ be a closed orientable non-elementary 2-orbifold of negative Euler characteristic. Assume that $O$ is neither $\bS^2(2,2,2,2,p)$, $p\ge 2$ nor $\bS^2(2,2,p,q)$, $\frac{1}{p}+\frac{1}{q}< 2$. Then for any $i=1,2,\cdots, n$, the following set
    \[
    \{[\rho]\in \Hit_n(O)\,|\, \lambda_i (\rho(x))+\lambda_{n-i+1}(\rho(x)) \ne 0 \text{ for all  regular }x\notin \ker (\pi_1^{\mathrm{orb}}(O)\to \pi_1^{\mathrm{orb}}(O)^{\mathrm{ab}}) \}
    \]
    is generic in $\Hit_n(O)$.
\end{theorem}
Here, we say that a subset of an analytic manifold $M$ is $\emph{generic}$ if it is the complement of the union of countably many proper analytic subvarieties. Also $\bS^2(p_1,\cdots, p_k)$ is an orbifold whose underlying space is the 2-sphere $\bS^2$ with cone points of order $p_1, \cdots, p_k$.

The second generic property that we want to explore is closely related to the notion of strongly dense subgroup introduced by Breuillard, Green,  Guralnick, and Tao \cite{tao}. We first recall some definitions. We say that a subgroup $H$ of $G$ is \emph{strongly dense} if for any pair of non-commuting elements $x,y\in H$, the subgroup of $G$ generated by $x$ and $y$ is Zariski dense. Analogues to the Tits alternative, every finitely generated subgroup of a strongly dense subgroup of $G$ is either very simple (i.e., abelian) or very complicated (i.e., Zariski dense in $G$). Motivated from this, we say that a representation $\rho:\Gamma\to G$ of a finitely generated group $\Gamma$ into $G$ is \emph{strongly dense} if $\rho(x)$ and $\rho(y)$ generate a Zariski dense subgroup of $G$ for any pair of non-commuting elements $x,y\in \Gamma$.

One may observe that every Fuchsian representation of a surface group into $\PSL_2(\bR)$ is strongly dense. For a higher rank $G$, on the other hand, 
we can produce many Zariski dense $G$-Hitchin representations which are not strongly dense by amalgamating a Zariski dense Hitchin representation and a Fuchsian representation along a separating simple closed curve on a surface. Despite this, it is reasonable to believe that generic $G$-Hitchin representations are strongly dense. Indeed, when $G=\PSL_n(\bR)$, Long, Reid, and Wolff \cite{long} recently proved that the set of strongly dense $\PSL_n(\bR)$-Hitchin representations of a surface $S$ is generic in $\overline{\Hit}_n(S)$. We prove a similar result for Hitchin representations of orbifold groups.

\begin{theorem}\label{thm:main2}
     Let $O$ be a closed orientable non-elementary orbifold of negative Euler characteristic. Assume that $O$ has at most one even order cone point. Let $G$ be a split real form of a complex simple adjoint Lie group such that the Weyl group of $G$ contains $-1$. Then the set of strongly dense $G$-Hitchin representations is generic in $\overline{\Hit}_G(O)$. 
\end{theorem}

The restriction on underlying orbifolds is essential; if $O$ has two even order cone points, one can find a pair of non-commuting elements $s_1$ and $s_2$ in $\pi_1^{\mathrm{orb}}(O)$ whose orders are even, say, $2n$ and $2m$ respectively. Then $s_1^{n}$ and $s_2^m$ are order two non-commuting elements in $\pi_1^{\mathrm{orb}}(O)$. However, $\langle \rho(s_1 ^n), \rho(s_2^m)\rangle $ is always isomorphic to $\mathbb{Z}\rtimes \mathbb{Z}_2$ for any Hitchin representation $\rho$, which cannot be Zariski dense in $G$. 

Theorem~\ref{thm:main2} shows that split real forms of types $B_n$, $C_n$, $D_{2n}$, $E_7$, $E_8$, $F_4$ and $G_2$ contain plenty of strongly dense discrete orbifold subgroups of various Euler characteristics. This can be seen as a partial generalization of \cite[Theorem~4.5]{tao} and \cite[Corollary~8.3]{breuillard}, where they constructed strongly dense non-abelian free subgroups and closed surface subgroups in the complexification  $G^\mathbb{C}$ respectively. We remark that our result cannot be derived from \cite[Theorem~8.1 and Theorem~8.5]{breuillard} because the representation variety $\operatorname{Hom}(\pi_1^{\mathrm{orb}}(O),G^\mathbb{C})$ is not irreducible in general. 

Our theorem does not cover Lie groups of types $A_n$, $E_6$, and $D_{2n+1}$. However we still believe that these Lie groups also contain strongly dense orbifold fundamental groups.

Proofs of Theorem~\ref{thm:genericeigenvalue} and Theorem~\ref{thm:main2} are based on the following Theorem~\ref{thm:genericjordan}.  

\begin{theorem}\label{thm:genericjordan}
    Let $O$ be a closed orientable non-elementary orbifold of negative Euler characteristic other than $\bS^2(2,2,2,2,p)$, $p\ge 2$ or $\bS^2(2,2,p,q)$, $\frac{1}{p}+\frac{1}{q}< 2$. Let $G$ be the split real form of a complex simple adjoint Lie group. Let $\lambda:G\to \overline{\mathfrak{a}^+}$ be the Jordan projection. 
    
    \begin{itemize}
    \item[(i)]  For an arbitrary $\alpha\in \mathfrak{a}^*\setminus\{0\}$ and $c\in \bR$, the following set
        \[
    \{[\rho]\in \Hit_G(O)\,|\, \alpha(\lambda(\rho(x))) \ne c \text{ for all regular  }x\notin \ker (\pi_1 ^{\mathrm{orb}}(O)\to \pi_1^{\mathrm{orb}}(O)^{\mathrm{ab}}) \}
    \]    
    is generic in $\Hit_G(O)$.
    \item[(ii)] Assume furthermore that $\alpha$ satisfies $\alpha=-w_0\alpha$ for the longest element $w_0$ of the Weyl group of $G$. Then 
            \[
    \{[\rho]\in \Hit_G(O)\,|\, \alpha( \lambda (\rho(x)) ) \ne c \text{ for all regular }x\in \pi_1 ^{\mathrm{orb}}(O)\setminus\{1\} \}
    \]
    is generic in $\Hit_G(O)$.
    \end{itemize}
\end{theorem}
For the definition of the Jordan projection, consult Section~\ref{sec:background}. We also remark that the condition in (ii) is always satisfied if $G$ is of type $B_n$, $C_n$, $D_{2n}$, $E_7$, $E_8$, $F_4$, or $G_2$ as for these cases the longest element of the Weyl group of $G$ acts as $-1$ on $\mathfrak{a}$. 

Once we have Theorem~\ref{thm:genericjordan}, Theorem~\ref{thm:genericeigenvalue} immediately follows since, as illustrated in Example~\ref{ex:psl}, the Jordan projection $\lambda(\rho(x))$ of a $\PSL_n(\bR)$-Hitchin representation $\rho$ is more or less taking the log of the diagonal matrix that is conjugate to $\rho(x)$. 

The milestone of this paper is to use the Goldman product formula as a variation formula. This new perspective, combined with the flow lifting theorem (Theorem~\ref{prop:lifting}), paves the way to our main results.

\subsection*{Organization} Section~\ref{sec:background} provides a quick review on Hitchin representations and relevant Lie theory.  In Section~\ref{sec:goldman}, we introduce Goldman functions and state the Goldman product formula, which will be the key player of our proof. In Section~\ref{sec:lifting}, we will show that, under a finite covering, a Goldman flow along an oriented simple closed curve lifts to a Goldman flow along a multi-curves. This allows us to apply the Goldman product formula to more general situations. Then proofs of Theorem~\ref{thm:genericjordan} and Theorem~\ref{thm:main2} will be given in Section~\ref{sec:proof}. Finally, in Section~\ref{sec:discussion}, we discuss related problems focusing on $\Hit_3(S)$.

\subsection*{Acknowledgment} The author learned Question~\ref{qu:middleeigen} from Michael Zshornack during the CIRM introductory school to the IHP program ``Group Actions and Rigidity: Around the Zimmer Program" in April 2024. 
The author thanks the organizers and CIRM for hosting the wonderful conference. Also the author is grateful to  Michael Zshornack, Gye-Seon Lee and Yves Benoist for helpful discussions and suggestions. 

The author was supported by the BK21 SNU Mathematical Sciences Division. 

\section{Background}\label{sec:background}
In this section, we quickly review orbifolds, Hitchin representations and related Lie theory.  Standard references for these subjects are \cite{knapp, eberlein,Cbook,porti}. \cite{Benoist16} also discusses Jordan projections in detail including many useful properties and applications.

\subsection{Orbifolds}
An $n$\emph{-orbifold} $O$ is a Hausdorff topological space $|O|$ together with a collection of  quadruples $\{(\widetilde{U}_i,\Gamma_i, \phi_i,V_i)\}$ where 
\begin{itemize}
    \item $\widetilde{U}_i$ is a simply-connected domain in $\bR^n$;
    \item $C:=\{V_i\}$ is an open covering of $O$ closed under the intersection. Namely, if $V_1,V_2\in C$ then $V_1\cap V_2\in C$;
    \item $\Gamma_i$ is a group acting on $\widetilde{U}_i$ smoothly and effectively;
    \item $\phi_i: \widetilde{U}_i/\Gamma_i \to V_i$ is a homeomorphism.
\end{itemize}
These charts need to be compatible to one another in the following sense. Suppose that two charts $(\widetilde{U}_1,\Gamma_1, \phi_1,V_1)$ and $(\widetilde{U}_2, \Gamma_2, \phi_2, V_2)$ overlap. Then there is a chart $(\widetilde{U}_{12},\Gamma_{12}, \phi_{12}, V_{12})$ such that $V_{12} = V_1\cap V_2$. We request that there is an injective homomorphism $\rho_i:\Gamma_{12} \to \Gamma_i$ and a transition diffeomorphism $f_i : \widetilde{U}_{12} \to \widetilde{U}_i$ such that the following diagram is commutative for each $i=1,2$:
\[
\xymatrix{
\widetilde{U}_{12} \ar[d]_{f_i}\ar[rr] & & \widetilde{U}_{12}/\Gamma_{12} \ar[r]^{\phi_{12}}& V_{12}\ar[d]^{\text{inclusion}}\\
\widetilde{U}_i \ar[r] &\widetilde{U}_i/\rho_i(\Gamma_{12}) \ar[r] &\widetilde{U}_i/\Gamma_i\ar[r]_{\phi_i}& V_i
}.
\]

We say that an orbifold $O$ is \emph{orientable} if the group actions associated to orbifold charts and all transition maps are orientation preserving; $O$ is \emph{closed} if $O$ does not have the orbifold boundary and the underlying space $|O|$ is compact.

For each point $x$ of a 2-orbifold $O$, one can find a local chart of the form $\mathbb{R}^2/G_x$ where $G_x$ is a finite group of isometries on $\mathbb{R}^2$. In this case, $G_x$ is called the \emph{local group} of $x$.  We say $x\in O$ is \emph{regular} if $G_x=\{1\}$ and \emph{singular} otherwise. The set of singular and regular points of $O$ will be denoted by $O_{\mathrm{sing}}$ and $O_{\mathrm{reg}}$ respectively.

In this paper, we will only focus on closed orientable 2-orbifolds of negative Euler characteristics. Such an orbifold $O$ has the underlying space $|O|$ a closed orientable surface and there are finitely many cone points in $|O|$ of order $p_1,\cdots, p_k$. Since the Euler characteristic of $O$ is negative, the integers $p_1, \cdots, p_k$  satisfy
\[
\chi(|O|) - \sum_{i=1} ^k \left( 1-\frac{1}{p_i}\right)<0.
\]
Such an orbifold will be denoted by $|O|(p_1,\cdots,p_k)$. All closed orientable 2-orbifolds of negative Euler characteristics are hyperbolic; they can be realized as a quotient orbifold $\mathbb{H}^2/\Gamma$, where $\Gamma$ is a discrete subgroup of $\operatorname{Isom}^+(\mathbb{H}^2)$ possibly with torsion.

\emph{Elementary orbifolds} are ones that cannot be decomposed further along a 1-suborbifold into orbifolds of negative Euler characteristics. The regular set of any non-elementary closed orientable orbifold of negative Euler characteristic contains an essential simple closed curve.

Every orbifold $O$ can be covered by its universal orbifold $\widetilde{O}$. The group of deck transformation groups of the universal covering $\widetilde{O}\to O$ is the \emph{orbifold fundamental group} $\pi_1 ^{\mathrm{orb}}(O)$.

In what follows, an \emph{oriented closed curve} in a 2-orbifold $O$ is a smooth immersion $\gamma:\bS^1=\bR/\mathbb{Z}\to |O|$ such that $|\gamma^{-1}(O_{\mathrm{sing}})|$ is finite and each point $t\in\gamma^{-1}(O_{\mathrm{sing}})$ is given a prescribed local lift of $\gamma|[t-\epsilon,t+\epsilon]$ to the universal cover of $O$. We denote by $|\gamma|$ the image of an oriented closed curve $\gamma$.

Let $\gamma$ be an oriented curve in $O$. Let $p\in |\gamma|$. If $p\in O_{\mathrm{reg}}$, choose an orbifold chart $U$ around $p$ with the trivial local group and a subinterval $[a,b]\subset \bR/\mathbb{Z}$ such that $\gamma([a,b])\subset U$. Then we homotope $\gamma|[a,b]$ inside $U$ relative to the end-points $\{\gamma(a),\gamma(b)\}$. If $p\in O_{\mathrm{sing}}$, we homotope the prescribed lift of $\gamma$ in the universal cover fixing the end-points and project it down to $O$. This process gives a new oriented curve which is said to be \emph{elementary homotopic} to $\gamma$. We say that two oriented closed curves are \emph{homotopic} if they are connected by a sequence of elementary homotopies. If the domain of $\gamma$ is given a distinguished base-point, we assume that every elementary homotopy of $\gamma$ fixes the base-point.  

Similar to usual fundamental groups, it is known that $\pi_1 ^{\mathrm{orb}}(O)$ is homeomorphic to the group of homotopy classes of oriented based closed curves. For a oriented  closed curve $\gamma$ with a base-point $t_0\in \bR/\mathbb{Z}$, we denote by $\gamma_{t_0}$ an element of $\pi_1^{\mathrm{orb}}(O, \gamma(t_0))$ represented by the oriented closed curve $\gamma$ with the base-point $t_0$. We say that an oriented closed curve $\gamma$ on a closed 2-orbifold $O$ is \emph{essential} if it is not homotopic in $O$ to a point or an orbifold boundary component of $O$. 

Since $O_{\mathrm{reg}}$ is a subspace of $|O|$, we have a natural homomorphism $\pi_1(O_{\mathrm{reg}})\to \pi_1^{\mathrm{orb}}(O)$ that is known to be surjective. Therefore it is sometimes convenient to regard an element $x\in \pi_1^{\mathrm{orb}}(O)$ as an oriented curve in $O_{\mathrm{reg}}$.

Let $\gamma$ and $\eta$ be two oriented closed curves on $O$. We define the set
\[
\gamma \sharp \eta = \{(t,s)\in \bS^1\times \bS^1\,|\,\gamma(t) = \eta(s) \text{ is a transverse intersection point}\}.
\]
By abusing language, we call each element of $\gamma \sharp \eta$ an \emph{intersection point} between $\gamma$ and $\eta$. Also it will be convenient to employ the following convention: given an element $p=(t_0,s_0)\in \gamma \sharp \eta$,  we write $\gamma_p := \gamma_{t_0}$ and $\eta_p := \eta_{s_0}$. Note that both are elements of $\pi_1^{\mathrm{orb}}(O,\gamma(t_0))=\pi_1^{\mathrm{orb}}(O,\eta(s_0))$. If the orbifold $O$ is oriented, we define  $\operatorname{sign} p\in \{1,-1\}$ for an intersection point $p=(t_0,s_0)\in \gamma\sharp \eta$ by the sign of the ordered bases $\{\dot{\gamma}(t_0), \dot{\eta}(s_0)\}$ of $\operatorname{T}_{\gamma(t_0)}O$.  The \emph{algebraic intersection number} $\ri(\gamma,\eta)$ is defined to be the signed sum 
\[
\ri(\gamma,\eta):=\sum_{p\in \gamma \sharp \eta} \operatorname{sign} p.
\]

We say that an element $x\in \pi_1^{\mathrm{orb}}(O)$ \emph{regular} is if its normalizer $N_{\pi_1 ^{\mathrm{orb}}(O)}(x)$ is an infinite cyclic group generated by $x$. An oriented closed curve in $O$  is regular if it represents a regular element of $\pi_1^{\mathrm{orb}}(O)$. 

Assume that $O$ is equipped with a hyperbolic structure. If an oriented closed curves $\gamma$ represents an infinite order element of $\pi_1^{\mathrm{orb}}(O)$, then $\gamma$ is freely homotopic to a closed geodesic or a geodesic full 1-suborbifold (see \cite{cgorb} for the definition). We say that two closed curves $\gamma$ and $\eta$ on $O$ \emph{intersect essentially} if their geodesic representatives have non-empty intersection. This definition in fact does not depend on the choice of a hyperbolic structure on $O$.

\begin{lemma}\label{lem:orbifold}
    Let $O$ be a closed non-elementary hyperbolic orbifold. Assume that $O$ is neither  $\bS^2(2,2,2,2,p)$, $p\ge 2$ nor $\bS^2(2,2,p,q)$, $\frac{1}{p}+\frac{1}{q}< 2$.
    \begin{itemize}
        \item[(i)] Any regular essential oriented simple closed curve can be homotoped inside $O_{\mathrm{reg}}$ to a regular oriented simple closed geodesics contained in  $O_{\mathrm{reg}}$. 
        \item[(ii)] An infinite order element $x\in \pi_1^{\mathrm{orb}}(O)$ is not regular if and only if there are order two elements $s_1$ and $s_2$ in $\pi_1^{\mathrm{orb}}(O)$ such that $x = s_1 s_2$. 
        \item[(iii)] Let $\gamma$ be a regular oriented closed curve. Then there is a regular simple closed curve $\eta$ intersecting $\gamma$ essentially. 
        \end{itemize}
\end{lemma}
\begin{proof}
As $O$ is given a hyperbolic structure, we may regard $\pi_1^{\mathrm{orb}}(O)$ as a discrete subgroup of $\operatorname{Isom}^+(\mathbb{H}^2)$. 

    (i) has been addressed in many places; see for instance \cite[Theorem 4.3]{cgorb}, \cite[Theorem 5.1]{tan} and \cite[Lemma 5.8]{baik2023}.

    (ii) Since $x$ has infinite order it is a hyperbolic element in $\operatorname{Isom}^+(\mathbb{H}^2)$. Note that elements in the normalizer of $x$ preserve the fixed points of $x$ in $\overline{\mathbb{H}^2}$. By the discreteness of $\pi_1^{\mathrm{orb}}(O)$, we know that each element $y$ in the normalizer of $x$ is either a power of $x$ or an order two element preserving two fixed points of $x$ in $\partial_\infty\mathbb{H}^2$. Hence, if $x$ is not regular, the normalizer of $x$ contains an order two element, say $s_1$. We then know that $s_1 x s_1 = x^{-1}$. Let $s_2 = s_1 x$. Since $s_2 ^2 = s_1 x s_1 x = x^{-1} x =1$, the forward implication is proved. The converse direction is a straightforward computation.
    
    (iii) It suffices to find a collection of simple regular closed geodesics whose complement consists of disks, annuli and disks with one cone point. 
    
    If $O$ is a sphere with six order two cone points, we take three simple regular closed curves as in Figure~\ref{fig:sixcone}. By (ii), these curves are regular. 
    
    For other 2-orbifolds $O$, we first construct $O'$ by splitting $O$ along pairwise disjoint geodesic full-1-suborbifolds joining pairs of order two cone points so that the $O'$ contains at most one order two cone point. The resulting orbifold $O'$ has totally geodesic boundary and its underlying space contains essential simple closed curves. It is a standard fact that one may find a collection of essential simple closed curves $\mathcal{S}$ in $(O')_{\mathrm{reg}}$ such that $(O')_{\mathrm{reg}}\setminus \mathcal{S}$ consists of disks, once punctured disks and annuli. Again, (ii) shows that these curves are regular and each component of $O\setminus \mathcal{S}$ is a disk, annulus or a disk with one cone point. 
    \end{proof}

    \begin{figure}[hbt]
        \centering
        \includegraphics[width=0.7\linewidth]{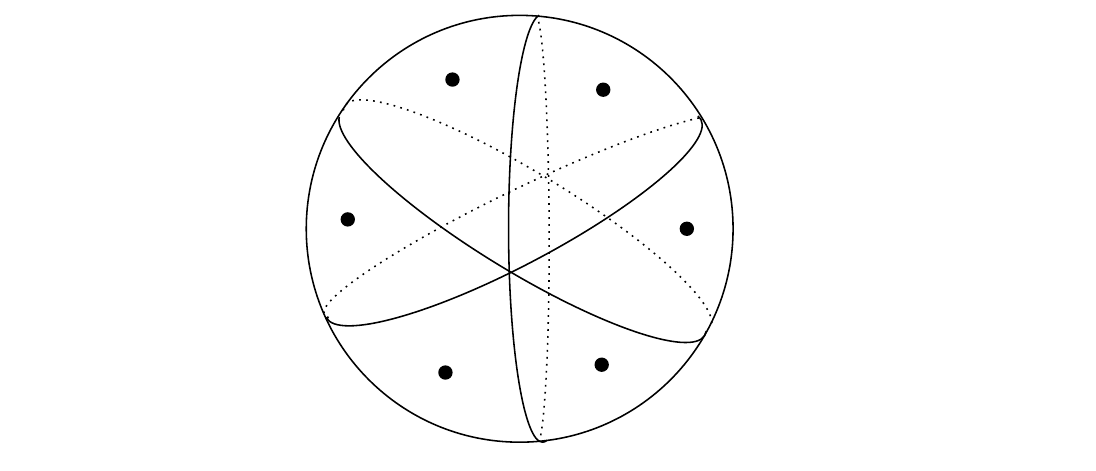}
        \caption{A sphere with six order two cone points. Black dots are order two cone points.}
        \label{fig:sixcone}
    \end{figure}

\subsection{Lie theory and Hitchin representations}
Let $G$ be a split real form of a complex simple Lie group $G^\mathbb{C}$ of adjoint type and let $\mathfrak{g}$ be its Lie algebra. In the Lie algebra $\mathfrak{g}$ of $G$, we choose a regular nilpotent element $E$. Then, we can always complete $E$ to a triple $\{E,F,H\}$ satisfying the $\mathfrak{sl}_2$-relations: $[H,E] = 2E$, $[H,F]=-2F$ and $[E,F]=H$. Such a principal $\mathfrak{sl}_2$-triple induces a representation $\mathfrak{sl}_2(\mathbb{C})\to \mathfrak{g}\otimes \mathbb{C}$ which can be integrated to a Lie group representation $\operatorname{PSL}_2(\mathbb{C})\to G^\mathbb{C}$. By  taking the real points,  one can find a Lie group representation $\iota_G:\PSL_2(\bR) \to G$. 

We  call a representation $\pi_1^{\mathrm{orb}}(O)\to G$ a \emph{Fuchsian representation} if it is of the form $\iota_G\circ \rho_{hyp}$ for some discrete faithful representation $\rho_{hyp}:\pi_1 ^{\mathrm{orb}}(O)\to \PSL_2(\bR)$. A representation $\rho:\pi_1^{\mathrm{orb}}(O)\to G$ is a $G$\emph{-Hitchin representation} if $\rho$ can be continuously deformed to a Fuchsian one. The set of $G$-Hitchin representations will be denoted by $\overline{\Hit}_G(O)$. We also define $\Hit_G(O) := \overline{\Hit}_G(O)/G$, the space of conjugacy classes of $G$-Hitchin representations. We will use the square bracket $[\rho]\in \Hit_G(O)$ to denote the conjugacy class of $\rho\in \overline{\Hit}_G(O)$. Topological properties of orbifold Hitchin components are studied in \cite{alessandrini}.

The Lie algebra $\mathfrak{g}$ of $G$ admits a $\Ad_G$-invariant non-degenerate symmetric bilinear form $\kf$. We have the Cartan involution $\theta:\mathfrak{g}\to \mathfrak{g}$, which makes the symmetric bilinear form 
\[
\kf_\theta(\cdot,\cdot) : = -\kf(\cdot, \theta(\cdot))
\]
positive definite on $\mathfrak{g}$. The Cartan involution $\theta$ decomposes $\mathfrak{g}$ into the direct sum $\mathfrak{g} = \mathfrak{k}\oplus \mathfrak{p}$, where $\mathfrak{k}$ is the +1 eigenspace of $\theta$ and $\mathfrak{p}$ is the $-1$ eigenspace of $\theta$. Hence, $\kf$ is positive definite on $\mathfrak{p}$ and negative definite on $\mathfrak{k}$.

Choose a maximal abelian subalgebra $\mathfrak{a}$ in $\mathfrak{p}$. Then we have the restricted root space decomposition
\[
\mathfrak{g} = \mathfrak{g}_0 + \bigoplus_{\beta\in \Pi} \mathfrak{g}_\beta
\]
where $\Pi\subset \mathfrak{a}^*\setminus\{0\}$ is the set of restricted roots and 
\[
\mathfrak{g}_\beta=\{X\in \mathfrak{g}\,|\, [A,X] = \beta(A)X\text{ for all }A\in \mathfrak{a}\}.
\]
If $G$ is real split, we have 
\[
\mathfrak{g}_0=\{X\in \mathfrak{g}\,|\,[A,X]=0\text{ for all }A\in \mathfrak{a}\} = \mathfrak{a}.
\]
This decomposition is orthogonal with respect to $\kf_\theta$.

We now define the Jordan projection. Choose an ordering on $\Pi$ and let $\Delta$ be the set of simple roots with respect to this order. Let 
\[
\mathfrak{a}^+=\{A\in \mathfrak{a}\,|\, \beta(A)>0 \text{ for all }\beta\in \Delta\}
\]
and let $\overline{\mathfrak{a}^+}$ be the closure of $\mathfrak{a}^+$. Any element $x\in G$ can be uniquely written as the product $x=x_k x_h x_u$ of commuting elements where $x_k$ is conjugate into the analytic subgroup of $\mathfrak{k}$, $x_h$ is in the conjugacy class of $\exp \overline{\mathfrak{a}^+}$ and $x_u$ is unipotent. We know that the conjugacy class of $x_h$ intersects  $\exp \overline{\mathfrak{a}^+}$ at a unique point. The \emph{Jordan projection} of $x$ is a unique element $\lambda(x)\in \overline{\mathfrak{a}^+}$ such that $\exp \lambda(x)$ is conjugate to $x_h$. This defines a map $\lambda:G\to \overline{\mathfrak{a}^+}$. 

We say that an element $x\in G$ is \emph{loxodromic} if $\lambda(x)\in \mathfrak{a}^+$. Let $Z_K(\mathfrak{a})$ be the centralizer of $\mathfrak{a}$ in the analytic subgroup of $\mathfrak{k}$. Then $x\in G$ is loxodromic if and only if $x$ is conjugate into $y\exp\mathfrak{a}^+$ for some $y\in Z_K(\mathfrak{a})$. We say $x\in G$ is \emph{purely loxodromic} if $x$ is conjugate into $\exp \mathfrak{a}^+$.

\begin{example}\label{ex:psl}
As a concrete example, take $G=\PSL_n(\bR)$. Then the Cartan involution $\theta$ and the invariant symmetric bilinear form $\kf$ may be taken to be $X\mapsto - X^{\operatorname{T}}$ and $\kf(X,Y) = \Tr(XY)$ respectively. An element $x\in \PSL_n(\bR)$ is loxodromic if, for some lift $x'\in \operatorname{SL}_n(\bR)$ of $x$, the eigenvalues of $x'$ are real and their absolute values are all distinct. $x\in \PSL_n(\bR)$ is purely loxodromic if $x$ can be lifted to $x'\in \operatorname{SL}_n(\bR)$ whose eigenvalues are positive and distinct. For an element $x\in \operatorname{GL}_n(\bR)$, let $\varepsilon_i(x)$ be (complex) eigenvalues of $x$ such that $|\varepsilon_1(x)|\ge |\varepsilon_2(x)|\ge\cdots \ge |\varepsilon_n(x)|$.  The Jordan projection of a loxodromic element $x$ is $\diag(\log |\varepsilon_1(x')|, \cdots, \log |\varepsilon_n(x')|)$ where $x'\in \operatorname{SL}_n(\bR)$ is a lift of $x$.   
\end{example}

\section{The Goldman Product Formula}\label{sec:goldman}
In this section we introduce the Goldman product formula which computes the Atiyah-Bott-Goldman symplectic form on Hamiltonian vector fields of a certain kind.

Let $G$ be a real semi-simple Lie group. and let $U\subset G$ be a non-empty open set which is invariant under conjugation in the sense that $x Ux^{-1} \subset U$ for all $x\in G$. A differentiable function $f:U\to \mathbb{R}$ is \emph{invariant} if $f(yxy^{-1}) = f(x)$ for all $x\in U$ and $y\in G$. We say that an invariant function $f:U\to \bR$ is \emph{homogeneous} if $f(x^k) = k f(x)$ for any $x\in U$ and any positive integer $k$ with $x^k\in U$.  

Given an invariant function $f$ we define the \emph{Goldman function}  $\widehat{f}:U \to \mathfrak{g}$ which associates to $x\in U$  a unique element $\widehat{f}(x)\in \mathfrak{g}$  satisfying
\[
\kf( \widehat{f}(x), Y) = \left.\frac{\rd}{\rd t}\right|_{t=0} f(x \exp tY) 
\]
for all $Y\in \mathfrak{g}$.

Let $f:U\to \bR$ be an invariant function such that the domain $U$ contains all purely loxodromic elements and let $O$ be a closed orientable orbifold of negative Euler characteristic. For a given infinite order element $x\in \pi_1 ^{\mathrm{orb}}(O)\setminus \{1\}$, we define the differentiable function $f_x:\Hit_G(O)\to \bR$ by  $f_x( [\rho] ) = f(\rho(x))$. This function is well-defined because $\rho(x)$ is always purely loxodromic; see Lemma~\ref{lem:generic}. Since $f$ is an invariant function, we know that $f_x$ depends only on the conjugacy class of $x$. Hence, we may say that $f_x$ is defined for the homotopy class of oriented essential closed curves representing $x$.

We are now ready to state the product formula by Goldman \cite[Theorem 3.5]{Goldman86}. The original statement is about symplectic geometry. At this moment, however, we only need the existence of Hamiltonian vector fields, which one may take it for granted. We will explain the construct of these Hamiltonian flows in Section~\ref{sec:lifting}.

\begin{theorem}[Goldman Product formula]\label{thm:goldmanproduct}
    Let $S$ be a closed oriented surface of genus $>1$. Let $f$ and $g$ be invariant functions. Let $x$ and $y$ be oriented essential closed curves in $S$. Assume that all intersections points of $x$ and $y$ are transverse double points. Let $\mathbb{X}_{g_y}$ be the Hamiltonian vector field associated to $g_y$ with respect to the Atiyah-Bott-Goldman symplectic form. Then, at any point $[\rho]\in \Hit_G(S)$, we have
    \[
    \mathbb{X}_{g_y} (f_x)|_{[\rho]}=\sum_{p\in x\sharp y} (\sign p ) \kf( \widehat{f}(\rho(x_p)), \widehat{g}(\rho(y_p))).
    \]
\end{theorem}
We remark that the transverse double intersection condition in the theorem is not essential. If $x$ and $y$ have a multiple intersection, we can always resolve it to a collection of double intersections by perturbing $x$ and $y$ near the intersection point.

Goldman's motivation for promoting Theorem~\ref{thm:goldmanproduct} was to compute the Atiyah-Bott-Goldman symplectic form between two Hamiltonian vector fields. Besides its origin, we will focus on the aspect that Theorem~\ref{thm:goldmanproduct} provides the first variation formula for functions along Hamiltonian flows. This viewpoint sheds light on the way toward our main theorems.

In the remaining part of this section, we compute the Goldman function for a certain class of invariant functions. 

We begin with a simple observation:
\begin{lemma}\label{lem:conjugation}
    Let $\widehat{f}:U\to \mathfrak{g}$ be the Goldman function of an invariant function $f$. Then for any $x\in U$ and $y\in G$, we have $\widehat{f}(yxy^{-1}) = \Ad_{y} \widehat{f}(x)$.
\end{lemma}

Let $\lambda:G\to \overline{\mathfrak{a}^+}$ be the Jordan projection. By definition, we have $\lambda(yxy^{-1}) = \lambda(x)$ for all $x,y\in G$. Hence, given a non-trivial linear functional $\alpha\in \mathfrak{a}^*$, the function $\alpha \circ \lambda :G\to \bR$ satisfies $\alpha\circ \lambda (yxy^{-1}) = \alpha \circ \lambda (x)$ for all $x,y\in G$. 

We are interested in the case when the domain $U\subset G$ is given by the set of purely loxodromic elements. Defined on $U$, we are concerned with the function $\alpha\circ \lambda$. To say that $\alpha\circ \lambda$ is an invariant function, we need to verify that $U$ is open and that $\lambda$ is at least differentiable. As we could not find an explicit statement from existing literature, we decide to give their proofs in Appendix~\ref{app:open}. Here we just state the results.  

\begin{lemma}\label{lem:openness} Let $G$ be a split real form of a complex simple adjoint Lie group. Let $\lambda:G\to \overline{\mathfrak{a}^+}$ be the Jordan projection and let $U$ be the set of purely loxodromic elements. Then
\begin{itemize}
    \item[(i)] $U$ is open,
    \item[(ii)] $\lambda$ is real analytic on $U$.
\end{itemize}
\end{lemma}

\begin{lemma}\label{lem:goldmanfunctioncompute}
Let $G$ be a split real form of a complex simple adjoint Lie group. Let $U$ be the set of purely loxodromic elements in $G$. Let $f:U\to \bR$ be given by $f:=\alpha \circ \lambda$. Then,
\begin{itemize}
    \item[(i)] $f$ is a real analytic invariant function, 
    \item[(ii)]$f$ is homogeneous,
    \item[(iii)] for an element $x\in \exp  \mathfrak{a}^+$, we have
    \[
    \widehat{f}(x)= \alpha^\vee,\quad\text{ and }\quad  \widehat{f}(x^{-1} ) = w_0 \alpha^\vee.
    \]
    where $\alpha^\vee\in \mathfrak{a}$ is a unique element in $\mathfrak{a}$ satisfying $\kf(\alpha^\vee, Y) = \alpha(Y)$ for all $Y\in \mathfrak{a}$ and $w_0$ is the longest element of the Weyl group of $G$.
\end{itemize}
\end{lemma}
\begin{proof}
    (i) This follows from Lemma~\ref{lem:openness}.

    (ii) Let $x\in U$. One can find $y\in G$ such that $yxy^{-1}=\exp X$ for some $X\in \mathfrak{a}^+$. For any positive integer $k\ne 0$,  we observe
    \[
    \lambda((\exp X)^k) = kX.
    \]
    Thus,  we have
    \[
    \lambda(x^k)=\Ad_{y^{-1}} \lambda((\exp X)^k)=k\Ad_{y^{-1}} \lambda(\exp X)=k\lambda (x).
    \]
    (iii) Write $x^{\pm 1}= \exp \pm X$ for some $X\in \mathfrak{a}^+$. Let $Y\in \mathfrak{g}$ be an arbitrary element. By the Baker-Campbell-Hausdorff formula, we have
    \begin{equation}\label{eq:expansion}
    x^{\pm 1} \exp tY = \exp \pm X \exp tY =\exp\left( \pm X+t\frac{\ad_{\pm X}}{1-\re^{-\ad_{\pm X}}}Y+ O(t^2)\right).
    \end{equation}
     Let 
    \[
    \mathfrak{g} = \mathfrak{g}_0\oplus \bigoplus_{\beta\in \Pi} \mathfrak{g}_\beta
    \]
    be the restricted root space decomposition of the Lie algebra $\mathfrak{g}$ of $G$. Since $G$ is a split real form we have $\mathfrak{g}_0 = \mathfrak{a}$. Decompose $Y$ into the sum of restricted root vectors; $Y=Y_0 + \sum_{\beta\in \Pi} Y_\beta$. Since $X\in \mathfrak{a}^+$, we have $[X,Y_\beta]=\beta(X)Y_\beta$ for every restricted root $\beta\in \Pi$.  Hence, the argument of the rightmost exponential in (\ref{eq:expansion}) becomes
    \[
    \pm X+t Y_0+t\sum_{\beta\in \Pi}\frac{\beta(\pm X)}{1-\re^{-\beta(\pm X)}}Y_\beta+O(t^2).
    \]
    Let
    \[
    Z_{\pm}:=  Y_0+\sum_{\beta\in \Pi}\frac{\beta(\pm X)}{1-\re^{-\beta(\pm X)}}Y_\beta.
    \]
    For a small enough $t$, (\ref{eq:expansion}) remains purely loxodromic. Also by Lemma~\ref{lem:openness}, for all sufficiently small $t$, we can find an analytic path $g_t\in G$ such that $g_t (x\exp tY) g_t^{-1} \in \exp \mathfrak{a}^+$. In the case of $x^{-1}$, there is a longest element $w_0$ in the Weyl group $N_{\mathfrak{k}}(\mathfrak{a})/Z_{\mathfrak{k}}(\mathfrak{a})$ such that $g_t w_0 (x^{-1} \exp t Y) w_0 ^{-1} g_t^{-1} \in \exp \mathfrak{a}^+$. Since $t\mapsto g_t$ is analytic with $g_0=1$, we expand $g_t$ near $t=0$ and write $g_t=\exp(tQ+O(t^2))$ for some $Q\in \mathfrak{g}$. Then we have  
    \begin{align*}
g_t (x\exp tY) g_t^{-1}&= \exp\left( \Ad_{g_t}( X+ tZ_+ +O(t^2))\right)\\
    &= \exp \left(X+tZ_+ + t[Q,X]+ O(t^2) \right)
    \end{align*}
    and
\begin{align*}
g_t w_0 (x^{-1} \exp tY) w_0 ^{-1} g_t^{-1}&= \exp\left( \Ad_{g_t}\Ad_{w_0} (- X+ tZ_- +O(t^2))\right)\\
    &= \exp \left(- \Ad_{w_0} X+t\Ad_{w_0} Z_- - t[Q,\Ad_{w_0} X]+ O(t^2) \right).
    \end{align*}

    Recall that both $X$ and $-\Ad_{w_0} X$ are members of $\mathfrak{a}^+$. By the construction of $g_t$, the first order terms, 
    \[
    tZ_+ +t[Q,X],\quad \text{ and }\quad t\Ad_{w_0}Z_- -t[Q,\Ad_{w_0} X]
    \]
    should be contained in $\mathfrak{a}^+$. Since $X$ and $\Ad_{w_0} X$ belong to  $\mathfrak{a}$, both $[Q,X]$ and $[Q,\Ad_{w_0} X]$ do not have  $\mathfrak{g}_0$-factors in the restricted root space decomposition. This forces
\begin{align*}
[Q,X]&=-\sum_{\beta\in \Pi}\frac{\beta(X)}{1-\re^{-\beta(X)}}Y_\beta\\
[Q,\Ad_{w_0} X]&=\sum_{\beta\in \Pi}\frac{\beta(-X)}{1-\re^{-\beta(-X)}}\Ad_{w_0}Y_\beta.
\end{align*}

Therefore, we know that  
\begin{align*}
    \lambda (x \exp tY) &=  X+t Y_0 +O(t^2)\\
    \lambda (x^{-1} \exp tY) &= -\Ad_{w_0} X+t\Ad_{w_0} Y_0 +O(t^2)
\end{align*} and that
    \begin{align*}
       \left.\frac{\rd}{\rd t}  \right|_{t=0}f(x\exp tY)&= \alpha ( Y_0)\\
           \left.\frac{\rd}{\rd t}  \right|_{t=0}f(x^{-1}\exp tY)&= \alpha (\Ad_{w_0} Y_0).
    \end{align*}
    Since the decomposition $Y=Y_0 + \sum_{\beta\in \Pi} Y_\beta$ is orthogonal with respect to the positive definite pairing $\kf_\theta$, we get 
    \[
    \kf(\alpha^\vee, Y ) = -\kf( \theta \alpha^\vee, Y )=\kf_\theta ( \alpha^\vee, Y)=\kf_\theta(\alpha^\vee,Y_0)=\kf(\alpha^\vee,Y_0)= \alpha( Y_0).
    \]
    Therefore, $\widehat{f}(x) = \alpha^\vee$. Similarly, we have $\widehat{f}(x^{-1}) = w_0 \alpha^\vee$.
\end{proof}

\section{Lifting Goldman Flows}\label{sec:lifting}
In this section, we investigate how Goldman flows behavior under a finite covering.

Let $O$ be a closed orientable non-elementary 2-orbifold with negative Euler characteristic and let $f:U\to G$ be an invariant function. Let $x$ be an oriented essential simple closed curve in $O_{\mathrm{reg}}$. Through the natural map $\pi_1(O_{\mathrm{reg}})\to \pi_1 ^{\mathrm{orb}}(O)$, $x$ represents an infinite order element in $\pi_1^{\mathrm{orb}}(O)$.

Goldman \cite{Goldman86} constructed an explicit Hamiltonian flow, called a \emph{Goldman flow}, associated to a given Hamiltonian vector field of the form $\mathbb{X}_{f_x}$. In \cite{CJK20,choi2023}, a multi-curve generalization of Goldman flows using graph of groups was studied. As we need Goldman flows along a multi-curve, let us quickly recall them. One may refer to \cite{serre} for more details on graph of groups.

Let $\mathbf{x}=\{x_1, \cdots, x_k\}$ be a set of pairwise non-isotopic oriented essential simple closed curves in $O_{\mathrm{reg}}$. We also regard $\mathbf{x}$ as a collection of oriented closed curves in $O$ via the inclusion $O_{\mathrm{reg}}\to O$.   Let $O_1,\cdots, O_d$ be the completions of connected components of $O\setminus (\bigcup_{i=1} ^k |x_i|)$. We first construct a graph of groups $\mathcal{G}$.  As a directed graph, $\mathcal{G}$ has the vertex set $V(\mathcal{G})=\{v_1,v_2,\cdots, v_d\}$ where $v_i\in O_i$ are regular points in $O_i$. We join $v_i$ and $v_j$ (possibly $i=j$) by an edge $e_k$ if and only if $O_i$ and $O_j$ are glued along some $|x_k|$.  We orient $e_k$  so that it passes $|x_k|$ from the left-side of $|x_k|$ to the right-side of $|x_k|$. Note that we can distinguish the left-side and the right-side of $|x_k|$ since $x_k$ is oriented and $O$ is orientable. Assign to each vertex $v_i$ the group $\pi_1^{\mathrm{orb}}(O_i,v_i)$. For each edge $e_k$, we give the group $\langle x_k \rangle$. We regard $\langle x_k\rangle$ as a subgroup of $\pi_1^{\mathrm{orb}}(O_i,v_i)$ by conjugating $x_k$ with the arc $e_k\cap O_i$ if $v_i$ is the head of the edge $e_k$.

Choose a maximal tree $\mathcal{T}$ of $\mathcal{G}$ such that $v_1\in V(\mathcal{T})$. Then it is a standard fact that $\pi_1^{\mathrm{orb}}(O,v_1)$ is isomorphic to the fundamental group $\pi_1(\mathcal{G},\mathcal{T})$ of the graph of groups. Recall that $\pi_1(\mathcal{G},\mathcal{T})$ is generated by $\pi_1^{\mathrm{orb}}(O_1,v_1), \cdots, \pi_1^{\mathrm{orb}}(O_k,v_k)$ together with edges in $E_\mathcal{T}(\mathcal{G}):=E(\mathcal{G})\setminus E(\mathcal{T})$. The isomorphism $\mathfrak{j}:\pi_1(\mathcal{G},\mathcal{T})\to \pi_1^{\mathrm{orb}}(O,v_1)$  maps $\gamma\in \pi_1^{\mathrm{orb}}(O_i,v_i)\subset \pi_1(\mathcal{G},\mathcal{T})$ to the element $\mathfrak{j}(\gamma)=e \gamma \overline{e}$ where $e$ is a unique edge path in $\mathcal{T}$ joining $v_1$ and $v_i$ with its reversal $\overline{e}$. Also it maps $c\in E_\mathcal{T}(\mathcal{G})$ to the element $g c h$ where $g$ is an edge path in $\mathcal{T}$ from $v_1$ to the tail of $c$  and $h$ is an edge path in $\mathcal{T}$ from the head of $c$ to $v_1$. 

To define the Goldman flow, we use the following notations. For $e\in E(\mathcal{T})$, let $H(e)$ be the head segment of $e$ defined by the connected component of $\mathcal{T}\setminus e$ containing the head vertex of $e$. The tail segment $T(e)$ of $e$ is defined similarly. Now given $e_i\in E(\mathcal{T})$,  the Goldman flow $\mathcal{F}^{f, x_i}_{t_i}(\rho)$ along $x_i$ is defined for $\rho\in \overline{\Hit}_G(O)$ by the conjugacy class of representation $\pi_1^{\mathrm{orb}}(O)\to G$ induced from the map on the generators
\[
\gamma \mapsto \begin{cases}
    \rho(\mathfrak{j}(\gamma)) & \gamma\in \pi_1^{\mathrm{orb}}(O_j,v_j) \text{ and }v_j\in V(T(e_i))\\
    \exp (-t\widehat{f}(\rho(\mathfrak{j}(x_i))))\rho(\mathfrak{j}(\gamma))  \exp (t\widehat{f}(\rho(\mathfrak{j}(x_i))))& \gamma\in \pi_1^{\mathrm{orb}}(O_j,v_j)  \text{ and }v_j\in V(H(e_i)) \\
    \exp (-t\widehat{f}(\rho(\mathfrak{j}(x_i))))\rho(\mathfrak{j}(e))  \exp (t\widehat{f}(\rho(\mathfrak{j}(x_i))))&\gamma = e\in E_\mathcal{T}(\mathcal{G}) \text{ joining two vertices of } H(e_i)\\
    \rho(\mathfrak{j}(e)) & \gamma = e\in E_\mathcal{T}(\mathcal{G}) \text{ joining two vertices of } T(e_i)\\
     \exp( -t\widehat{f}(\rho(\mathfrak{j}((x_i))))  \rho(\mathfrak{j}(e)) & \gamma= e\in E_\mathcal{T}(\mathcal{G})\text{ directing from } H(e_i)\text{ to } T(e_i)\\
    \rho(\mathfrak{j}(e)) \exp (t\widehat{f}(\rho(\mathfrak{j}(x_i)))) & \gamma=e\in E_\mathcal{T}(\mathcal{G})\text{ directing from } T(e_i)\text{ to } H(e_i)
\end{cases}
\]
using the universal property of $\pi_1(\mathcal{G},\mathcal{T})$. When $e_i\in E_\mathcal{T}(\mathcal{G})$, we similarly define $\mathcal{F}^{f,x_i}_t(\rho)$ by the representation extending the map on the generators
\[
\gamma \mapsto \begin{cases}
    \rho(\mathfrak{j}(e_i)) \exp(-t\widehat{f}(\rho(\mathfrak{j}(x_i)))) &\text{ if }\gamma = e_i\\
        \rho(\mathfrak{j}(\gamma)) & \text{ otherwise } 
\end{cases}.
\]
We have defined a flow on $\overline{\Hit}_G(O)$.  The actual flow, also denoted by $\mathcal{F}^{f,x_i}_t$, on $\Hit_G(O)$ is given by $[\rho]\mapsto [\mathcal{F}^{f,x_i}_t(\rho')]$ where $\rho'$ is a representative of $[\rho]$. One can check that this is well-defined. 

Since $|x_i|$ are pairwise disjoint, $\mathcal{F}^{f,x_i}_{t_i}$ and $\mathcal{F}^{f,x_j}_{t_j}$ commute whenever $i\ne j$. Therefore, for $\mathbf{t}=(t_1,\cdots, t_k)$, the $k$-parameter family of representations, $\mathcal{F}^{f,\mathbf{x}}_{\mathbf{t}}([\rho]) := \mathcal{F}^{f,x_1}_{t_1} \circ \cdots \circ \mathcal{F}^{f,x_k}_{t_k}([\rho])$ is well-defined and gives a smooth flow on $\Hit_G(O)$. We call $\mathcal{F}^{f,\mathbf{x}}_{\mathbf{t}}$ the \emph{Goldman flow} along $\mathbf{x}$. Indeed it is the Hamiltonian flow associated to the Hamiltonian vector field $\mathbb{X}_{f_{x_1}}+\cdots +\mathbb{X}_{f_{x_k}}$.

Now we state the main result of this section. Let $O$ be an orientable closed non-elementary 2-orbifold with negative Euler characteristic and let $x:\bS^1\to O_{\mathrm{reg}}$ be an oriented essential regular simple closed curve. Let $q: \widetilde{O}\to O$ be a finite normal covering. The covering map $q$ induces the inclusion $q_*:\pi_1^{\mathrm{orb}}(\widetilde{O}, \widetilde{p}) \to \pi_1^{\mathrm{orb}}(O,q(\widetilde{p}))$ as well as the smooth map, denoted by $q_!:\Hit_G(O) \to \Hit_G(\widetilde{O})$, defined by  $[\rho] \mapsto [\rho\circ q_*]$. Let $c_1, \cdots, c_d$ be the connected components of $q^{-1}(|x|)$. Since $x$ is regular, $c_1,\cdots,c_d$ are mutually non-isotopic. Let $\Delta$ be the deck group of $\widetilde{O}\to O$ and let $\operatorname{Stab}(c_1)$ be the stabilizer of $c_1$ in $\Delta$. Write $m:=|\operatorname{Stab}(\widetilde{x})|$. Since $q|c_i:c_i\to |x|$ is a $m$-fold covering, we find a parametrization $\widetilde{x}_i:\bS^1 \to |\widetilde{O}|$ of $c_i$ such that $q\circ \widetilde{x}_i(t) = x(t^m)$, where we regard $\bS^1$ as the set of unit complex numbers in $\mathbb{C}$.

\begin{theorem}\label{prop:lifting}
In the above situation, we have 
    \begin{equation}\label{eq:multitwist}
    q_!(\mathcal{F}_t ^{f,x}([\rho])) = \mathcal{F}_{t} ^{f, \widetilde{x}_1}\circ \cdots \circ \mathcal{F}_{t} ^{f,\widetilde{x}_d}(q_!([\rho])).
    \end{equation}
\end{theorem}

The main part of the proof is to construct a nice graph of groups for $\pi_1^{\mathrm{orb}}(\widetilde{O})$. Once we have done this, we just carefully read off the definition of the Goldman flow to show that the actions of both sides of (\ref{eq:multitwist}) agree on generators of $\pi_1^{\mathrm{orb}}(\widetilde{O})$. The proof is, by nature, tedious and repetitive. We split the proof into two cases depending on whether $x$ is separating or not, and omit some repeating computations. 

\begin{proof}[Proof of Theorem~\ref{prop:lifting}: Non-Separating case]
Assume that $x$ is non-separating. We know that the connected components of $q^{-1}(|x|)$ are $|\widetilde{x}_1|,\cdots, |\widetilde{x}_d|$. Let $\widetilde{O}_1$, $\widetilde{O}_2,\cdots$, $\widetilde{O}_n$ be the connected components of $\widetilde{O}\setminus q^{-1}(|x|)$. Each $\widetilde{O}_i$ is a covering space over $O_0:=O\setminus |x|$. 

We will construct a particular graph of groups for $\pi_1^{\mathrm{orb}}(\widetilde{O})$ that is compatible with the covering structure. We first define an underlying graph of $\mathcal{G}$ embedded in $\widetilde{O}$. Fix a base-point $v\in O\setminus |x|$ and an oriented loop $c$ based at $v$ with the algebraic intersection $\ri(c,x)=1$ so that $\pi_1^{\mathrm{orb}}(O,v)$ is decomposed into the HNN-extension $\pi_1^{\mathrm{orb}}(O_0,v)*_c$. We choose a regular point $v_i\in \widetilde{O}_i$ for each $i\in\{1,2,\cdots,n\}$ such that $q(v_i) = v$. They are vertices for $\mathcal{G}$. Suppose that $\widetilde{O}_i$ and $\widetilde{O}_j$ share a boundary $|\widetilde{x}_k|$ for some $k$ and $\widetilde{O}_i$ is on the left-side of $|\widetilde{x}_k|$. Then there is a unique oriented arc $e_k$ from $v_i$ to $v_j$ such that $q(e_k)=\eta_ k c \xi_k$ for some $\eta_k,\xi_k \in \pi_1^{\mathrm{orb}}(O_0)$. We join $v_i$ and $v_j$ by using this $e_k$ as the oriented edge. We give the vertex group $\pi_1^{\mathrm{orb}}(\widetilde{O}_i,v_i)$ to $v_i$ and the edge group $\langle \widetilde{x}_k\rangle$ to an edge $e_k$.   We implicitly understand $\widetilde{x}_k$ as the element $\overline{\alpha} \widetilde{x}_k \alpha$ in $\pi_1^{\mathrm{orb}}(\widetilde{O}_i,v_i)$, where $v_i$ is the head of $e_k$ and where $\alpha$ denotes the oriented arc $e_k\cap \widetilde{O}_i$. 
\begin{claim}\label{claim:maximaltree}
 The underlying directed graph of $\mathcal{G}$ admits a maximal \emph{rooted} tree, namely a maximal tree $\mathcal{T}$ such that each vertex of $\mathcal{T}$ has at most one inbound edge.   
\end{claim}
\begin{proof}
    Take any rooted tree $\mathcal{T}_0$ of $\mathcal{G}$. If $\mathcal{T}_0$ is not maximal, there is a vertex $w$ not in $V(\mathcal{T}_0)$ but has an edge $e_w$ to a vertex in $\mathcal{T}_0$. If this edge is inbound to $w$, we just add $w$ and $e_w$ enlarging $\mathcal{T}_0$ to a bigger rooted tree. Hence assume that the edge $e_w$ is outbound from $w$. By the construction of $\mathcal{G}$, the numbers of inbound and outbound edges are equal at each vertex. Therefore, we can find a longest directed trail terminating at $w$ and, since there are only finitely many vertices and edges, this trail must pass through a vertex in $\mathcal{T}_0$. Hence we find another vertex $w'$ and an edge $e_{w'}$ from some vertex of $\mathcal{T}_0$ into $w'$. Therefore, we can enlarge $\mathcal{T}_0$ by adding $w'$ and $e_{w'}$ as we did in the above case. By repeating this procedure, we obtain a maximal rooted tree $\mathcal{T}$.
\end{proof}
 
By Claim~\ref{claim:maximaltree}, we may find a maximal rooted tree $\mathcal{T}$ of $\mathcal{G}$. By renumbering if necessary, we assume that  $v_1\in V(\mathcal{T})$ is the root. 

As mentioned above, $\pi_1(\mathcal{G},\mathcal{T})$ is generated by $\pi_1^{\mathrm{orb}}(\widetilde{O}_1,v_1), \cdots, \pi_1^{\mathrm{orb}}(\widetilde{O}_n,v_n)$ together with edges $E_\mathcal{T}(\mathcal{G})$. Let $\mathfrak{j}:\pi_1(\mathcal{G},\mathcal{T})\to \pi_1^{\mathrm{orb}}(\widetilde{O},v_1)$ be the isomorphism. As described above, $\mathfrak{j}$ maps $z\in \pi_1^{\mathrm{orb}}(\widetilde{O}_i, v_i)$ to $e z \overline{e}$ where $e$ is a unique edge path from $v_1$ to $v_i$. 

We will see both sides of (\ref{eq:multitwist}) act identically on these generators. For instance let $z \in \pi_1^{\mathrm{orb}}(\widetilde{O}_i,v_i)$ for some $i$. Since $v_\ell \in V(\mathcal{T})$, there is a unique directed path, say,  $e:=e_1 e_2 \cdots e_\ell$ in $\mathcal{T}$ from  $v_1$ to $v_i$.  We know that $q_*(\mathfrak{j}(z)) = q(e) q(z) q(\overline{e})$. Moreover, $q(e_i) = \eta _i c \xi_i$ for some $\eta_i ,\xi_i \in \pi_1 ^{\mathrm{orb}}(O_0)$.  Following the definition of the Goldman flow, we know that $\mathcal{F}^{f,x}_t (\rho)(q_*(\mathfrak{j}(z)))$ is the conjugation of $\rho(q(z))$ by the element
\begin{multline*}
\rho(\eta_1) \rho (c) \exp \left(-t \widehat{f}(\rho(x_v))\right) \rho(\xi_1) \rho(\eta_2) \rho (c) \exp \left(-t \widehat{f}(\rho(x_v))\right) \rho(\xi_2)\cdots \\
\cdots \rho(\eta_\ell ) \rho (c) \exp \left(-t \widehat{f}(\rho(x_v))\right) \rho (\xi_\ell).
\end{multline*}
Here $x_v$ is defined as follow. Take a sub-arc $\gamma$ of $c$ joining the point $c\cap x$ and $v$. Then we define $x_v := \overline{\gamma} x \gamma$ and regard it as an element of  $\pi_1^{\mathrm{orb}}(O,v)$.  

Now we investigate how the right-hand-side of (\ref{eq:multitwist}) acts on $z$. Since $\mathcal{T}$ is rooted, we only need to compute,
\begin{equation}\label{eq:RHS}
\mathcal{F}_{t} ^{f, \widetilde{x}_1}\circ \mathcal{F}_{t} ^{f, \widetilde{x}_2}\cdots \circ \mathcal{F}_{t} ^{f,\widetilde{x}_\ell}(\rho\circ q_*)(\mathfrak{j}(z)),
\end{equation}
where (by renumbering if necessary) the curve $|\widetilde{x}_i|$ is the one that the edge $e_i$ crosses. Observe that for $r\in\{1,2,\cdots, \ell\}$,
\begin{align*}
q_*(\mathfrak{j}(\widetilde{x}_r)) &=q_*(e_1e_2\cdots e_r \overline{\alpha} \widetilde{x}_r\alpha \overline{e}_r \cdots \overline{e}_1 )\\
&= (q(e_1)\cdots q(e_{r-1})\eta_r c) x_v ^m ( q(e_1)\cdots q(e_{r-1})\eta_r c)^{-1}.
\end{align*}
Remember that we are interpreting $\widetilde{x}_r$ as the element $\overline{\alpha} \widetilde{x}_r \alpha$ in $\pi_1^{\mathrm{orb}}(\widetilde{O}_j,v_j)$, where $v_j$ is the head of $e_r$ and where $\alpha$ denotes the oriented arc $e_r\cap \widetilde{O}_j$. Hence,
\begin{multline}\label{eq:exp1}
\exp \left(-t \widehat{f}(\rho \circ q_*(\mathfrak{j}(\widetilde{x}_r)))\right)=\\
\rho(q(e_1)\cdots  q(e_{r-1})\eta_r c)\exp \left(-t \widehat{f}(\rho(x_v))\right) \rho(q(e_1)\cdots  q(e_{r-1})\eta_r c)^{-1}.
\end{multline}
By following the definition of the Goldman flow, we know that $(\ref{eq:RHS})$ is the conjugation of $\rho(q(z))$ by the element
\[
\exp \left(-t \widehat{f}(\rho \circ q_*(\mathfrak{j}(\widetilde{x}_1)))\right)\cdots \exp \left(-t \widehat{f}(\rho \circ q_*(\mathfrak{j}(\widetilde{x}_\ell)))\right).
\]
Combining (\ref{eq:exp1}) and 
\[
\rho(q_*(z_{v_i}))= \rho(q(e_1) \cdots q(e_\ell)) \rho(z_v) \rho(q(e_1) \cdots q(e_\ell))^{-1},
\]
we know that (\ref{eq:RHS}) is the same as $q_!(\mathcal{F}^{f,x}(\rho))(z_{v_i})$.

For a generator $e\in E_\mathcal{T}(\mathcal{G})$, a similar computation yields $q_!(\mathcal{F}^{f,x}(\rho))(\mathfrak{j}(e))=\mathcal{F}_{t} ^{f, \widetilde{x}_1}\circ \cdots \circ \mathcal{F}_{t} ^{f,\widetilde{x}_\ell}(\rho\circ q_*)(\mathfrak{j}(e))$.  \end{proof}

\begin{proof}[Proof of Theorem~\ref{prop:lifting}: Separating case]
We construct graph of groups for $\pi_1^{\mathrm{orb}}(\widetilde{O})$ as we did in the non-separating case. Denote by $L$ and $R$ the suborbifolds of $O$ in the left- and the right-side of $x$ respectively. First we take regular base points $v$ and $w$ on $L$ and $R$ respectively and choose an oriented simple path $c$ from $v$ to $w$ that crosses $x$ transversely once. In the covering $\{|\widetilde{x}_1|, \cdots, |\widetilde{x}_d|\}$ is the set of connected components of $q^{-1}(x)$ and $\widetilde{O}\setminus q^{-1}(x)=\widetilde{L}_1\cup\cdots \cup \widetilde{L}_{n_L}\cup \widetilde{R}_1\cup\cdots \cup \widetilde{R}_{n_R}$ where $\widetilde{L}_i$ are covering over $L$ and $\widetilde{R}_i$ are covering over $R$. We let $V(\mathcal{G})= \{v_1, \cdots, v_{n_L}, w_1, \cdots, w_{n_R}\}$  be the vertex set where $v_i\in \widetilde{L}_i$ and $w_i\in \widetilde{R}_i$. We join an oriented arc $e_k$ from $v_i$ to $w_j$ if $\widetilde{L}_i$ and $\widetilde{R}_j$ share a boundary $|\widetilde{x}_k|$. The arc $e_k$ requires $q(e_k) = \eta c\xi$ for some $\eta\in \pi_1(L,v)$ and $\xi\in \pi_1 ^{\mathrm{orb}}(R,w)$. The vertex groups and edge groups are defined in the obvious manner.

Choose any maximal tree $\mathcal{T}$ with $v_1\in V(\mathcal{T})$ and let $\mathfrak{j}:\pi_1(\mathcal{G},\mathcal{T})\to \pi_1^{\mathrm{orb}}(\widetilde{O},v_1)$ be the isomorphism. Recall that for $z\in \pi_1 ^{\mathrm{orb}}(\widetilde{L}_i,v_i)$, $\mathfrak{j}(z) = e z_v \overline{e}$, where $e$ is a path in $\mathcal{T}$ joining $v_1$ and $w_i$ and similarly for $v\in \pi_1^{\mathrm{orb}} (\widetilde{R}_i,w_i)$. 

When $x$ is separating, we do not have Claim~\ref{claim:maximaltree} which significantly slashed the amount of computation. Instead, we do some trick to detour this issue. 

Partition the edge set $E(\mathcal{T})$ into sets of edges 
\begin{align*}
    E_-(\mathcal{T})&:=\{e\in E(\mathcal{T})\,|\,e\text{ pointing toward }v_1\},\\
    E_+(\mathcal{T})&:=\{e\in E(\mathcal{T})\,|\,e\text{ pointing away from }v_1\},\text{ and }\\
    E_\mathcal{T}(\mathcal{G}) &:= E(\mathcal{G})\setminus E(\mathcal{T}).
\end{align*}
Define the dual flow for $\widetilde{x}_i$ intersecting an edge $e_i\in E_-(\mathcal{T})$ 
\[
\mathcal{F}^{f,-\widetilde{x}_i}_t(\rho):= \exp\left( t\widehat{f}(\rho\circ q_*(\mathfrak{j}(\widetilde{x}_i)))\right) \mathcal{F}^{f,\widetilde{x}_i}_t(\rho) \exp \left(-t\widehat{f}(\rho\circ q_*(\mathfrak{j}(\widetilde{x}_i)))\right)
\]
Since $\mathcal{F}^{f,-\widetilde{x}_i}_t$ and $\mathcal{F}^{f,\widetilde{x}_i}_t$  only differ by global conjugation, they behave identically on $\Hit_G(O)$. Thus we may instead compare the left-hand side of $(\ref{eq:multitwist})$ with
\[
\mathcal{F}^{f,\epsilon_1 \widetilde{x}_1}_{t}\circ \mathcal{F}^{f,\epsilon_2 \widetilde{x}_{2}}_{t} \circ \cdots \circ \mathcal{F}^{f,\epsilon_d \widetilde{x}_d}_{t}(q_! ([\rho]) )
\]
where $\epsilon_i=-1$ if the corresponding edge $e_i$ of $\widetilde{x}_i$ belongs to $E_-(\mathcal{T})$ and $\epsilon_i=1$ otherwise.

Let  $z\in \pi_1 ^{\mathrm{orb}}(L_i, v_i)$ be a generator of $\pi_1(\mathcal{G},\mathcal{T})$.  We have a unique edge path, say, $e_1 e_2 \cdots e_\ell$ in $\mathcal{T}$ from $v_1$ to $v_i$. Since $q(e_j)= \eta_j c \xi_j$ for some $\eta_j\in \pi_1^{\mathrm{orb}}(L,v)$ and $\xi_j\in \pi_1^{\mathrm{orb}}(R,w)$, one can write
\[
q_*(\mathfrak{j}(z))= (\eta_1 \xi_1 \eta_2 \xi_2 \cdots \eta_\ell \xi_\ell) q(z) (\eta_1 \xi_1 \eta_2 \xi_2 \cdots \eta_\ell \xi_\ell)^{-1}.
\]
Here, $\xi_i$ should be interpreted as $c\xi_i \overline{c}\in \pi_1^{\mathrm{orb}}(O,v)$. Therefore, $\mathcal{F}^{f,x}_t(\rho\circ q_*) (\mathfrak{j}(z))$ is the conjugate of $\rho(q(z))$ by the element
\begin{multline*}
\rho(\eta_1) \exp\left(-t\widehat{f}(\rho(x_v))\right) \rho(\xi_1)\exp\left(t\widehat{f}(\rho(x_v))\right)
\rho(\eta_2) \exp\left(-t\widehat{f}(\rho(x_v))\right) \rho(\xi_2)\exp\left(t\widehat{f}(\rho(x_v))\right)\cdots \\
\cdots \rho(\eta_\ell) \exp\left(-t\widehat{f}(\rho(x_v))\right) \rho(\xi_\ell)\exp\left(t\widehat{f}(\rho(x_v))\right).
\end{multline*}

To compute the right-hand side of (\ref{eq:multitwist}) evaluated at $z$, we observe that  $e_1,e_2,\cdots, e_\ell$ alternate $E_+(\mathcal{T})$ and $E_-(\mathcal{T})$. Therefore, it suffices to evaluate
\begin{equation}\label{eq:comp}
    \mathcal{F}^{f,\widetilde{x}_1}_t \circ \mathcal{F}^{f,-\widetilde{x}_2}_t \circ \cdots \circ \mathcal{F}^{f,\widetilde{x}_{\ell-1}}_t \circ \mathcal{F}^{f,-\widetilde{x}_\ell}_t (\rho\circ q_*) (\mathfrak{j}(z)).
\end{equation}
Because
\[
q_*(\mathfrak{j}(\widetilde{x}_j)) = (\eta_1 \xi_1 \cdots \xi_{j-1}\eta_j) x_v ^m(\eta_1 \xi_1 \cdots \xi_{j-1}\eta_j)^{-1}
\]
for $j\in \{1,2,\cdots, \ell\}$, we have
\begin{equation}\label{eq:exp}
\exp\left(\pm t \widehat{f}(\rho\circ q_*(\mathfrak{j}(\widetilde{x}_j)))\right) = \rho(\eta_1 \xi_1 \cdots \xi_{j-1}\eta_j) \exp \left( \pm t \widehat{f}(\rho(x_v))\right) \rho(\eta_1 \xi_1 \cdots \xi_{j-1}\eta_j)^{-1}. 
\end{equation}
Also we know that (\ref{eq:comp}) is the conjugation of
\[
\rho(\eta_1 \xi_1 \cdots \eta_\ell \xi_\ell) \rho(q(z)) \rho(\eta_1 \xi_1 \cdots \eta_\ell \xi_\ell)^{-1}
\]
by the element
\begin{multline*}
\exp\left( -t\widehat{f}(\rho\circ q_*(\mathfrak{j}(\widetilde{x}_1)))\right)\exp\left(t\widehat{f}(\rho\circ q_*(\mathfrak{j}(\widetilde{x}_2)))\right)\exp\left( -t\widehat{f}(\rho\circ q_*(\mathfrak{j}(\widetilde{x}_3)))\right)\cdots \\
\cdots\exp\left(-t\widehat{f}(\rho\circ q_*(\mathfrak{j}(\widetilde{x}_{\ell-1})))\right) \exp\left(t\widehat{f}(\rho\circ q_*(\mathfrak{j}(\widetilde{x}_\ell)))\right).
\end{multline*}
One can compute this element using (\ref{eq:exp}) to show that (\ref{eq:comp}) agrees with $\mathcal{F}^{f,x}_t(\rho\circ q_*) (\mathfrak{j}(z))$. 

Evaluations for generators $z\in \pi_1 ^{\mathrm{orb}}(R_i,w_i)$ and $e\in E_{\mathcal{T}}( \mathcal{G})$ are similar hence omitted.
\end{proof}

\section{Proofs of Main Theorems}\label{sec:proof}

By bringing results from the previous sections together, we can show our main theorem. First, we observe that Theorem~\ref{thm:goldmanproduct} can be generalized to an orbifold setting. Recall that we call $x\in \pi_1 ^{\mathrm{orb}}(O)$ \emph{regular} if its normalizer is infinite cyclic generated by $x$ itself.
\begin{lemma}\label{lem:generalderivative}
     Let $x$ be an oriented regular closed curve  on $O$ and let $y$ be an oriented regular simple closed curve on $O_{\mathrm{reg}}$. Let $f$ and $g$ be homogeneous invariant functions. Let $\mathcal{F}^{g,y}_t$ be the Goldman flow associated to the Hamiltonian vector field $\mathbb{X}_{g_{y}}$ and let $[\rho_t] = \mathcal{F}^{g,y}_t ([\rho])$. Then we have
    \[
    \left.\frac{\rd}{\rd t}\right|_{t=0} f_x([\rho_t]) = \sum_{p\in x\sharp y} (\sign p) \kf(  \widehat{f}(\rho(x_p)), \widehat{g}(\rho(y_p))).
    \]
\end{lemma}

\begin{proof}
    If $O$ is a surface, the lemma follows from the Goldman product formula. 

    For an orbifold $O$ we take a finite index normal surface covering $q: \widetilde{O}\to O$. Let $\Delta$ be the group of deck transformations of the covering $\widetilde{O}\to O$. 
    
    Let $c_1,\cdots, c_r$ be connected components of $q^{-1}(|y|)$. For each $i$, we can find a parametrization $\widetilde{y}_i:\bS^1\to \widetilde{O}$ such that $q\circ \widetilde{y}_i(t) = y(t^m)$, $t\in \bS^1\subset \mathbb{C}$, where $m=|\operatorname{Stab}_\Delta (c_1)|$.  Similarly, since $x$ is regular, we also find an oriented closed curves $\widetilde{x}_i:\bS^1\to \widetilde{O}$, $i=1,2,\cdots,s$, so that $q\circ \widetilde{x}_i (t) = x(t^l)$, $l=|\operatorname{Stab}_\Delta(|\widetilde{x}_1|)|$.

    We compute the following quantity at $q_! ([\rho])$
    \[
    \omega\left( \sum_{j=1} ^s \mathbb{X}_{f_{\widetilde{x}_j}},  \sum_{i=1}^{r} \mathbb{X}_{g_{\widetilde{y}_i}}\right),
    \]
    where $\omega$ is the Atiyah-Bott-Goldman symplectic form on $\Hit_G(\widetilde{O})$.
    
    On one hand we have
    \begin{align}
\omega\left( \sum_{j=1} ^s \mathbb{X}_{f_{\widetilde{x}_j}}, \sum_{i=1}^{r} \mathbb{X}_{g_{\widetilde{y}_i}}\right)&= \sum_{j=1} ^s \left. \frac{\rd}{\rd t} \right|_{t=0}f_{\widetilde{x}_j}(\mathcal{F}^{g,\widetilde{y}_1}_{t}\circ \cdots \circ \mathcal{F}^{g,\widetilde{y}_r}_{t} (q_!([\rho])))    \nonumber\\
&=\sum_{j=1} ^s \left. \frac{\rd}{\rd t}\right|_{t=0}f_{\widetilde{x}_j}( q_!(\mathcal{F}^{g,y}_t ([\rho]))) \nonumber\\
&= \sum_{j=1} ^s \left. \frac{\rd}{\rd t} \right|_{t=0}f( \mathcal{F}^{g,y}_t ([\rho])(x^l) )\nonumber\\
&= |\Delta| \left. \frac{\rd}{\rd t} \right|_{t=0}f( \mathcal{F}^{g,y}_t ([\rho])(x) ) \label{eq:lhs}
    \end{align}
    from the definition of Hamiltonian flow and Theorem~\ref{prop:lifting}. 

    On the other hand, we apply Theorem~\ref{thm:goldmanproduct} to the surface $\widetilde{O}$ to compute
    \begin{align*}
    \omega\left(  \sum_{j=1} ^s\mathbb{X}_{f_{\widetilde{x}_j}}, \sum_{i=1}^{r} \mathbb{X}_{g_{\widetilde{y}_i}}\right)&=\sum_{j=1}^s\sum_{i=1}^r \omega (\mathbb{X}_{f_{\widetilde{x}_j}}, \mathbb{X}_{g_{\widetilde{y}_i}})\\
    &=\sum_{j=1} ^s \sum_{i=1}^r \sum_{p\in \widetilde{x}_j \sharp \widetilde{y}_i} (\sign p) \mathsf{B}(\widehat{f}(\rho(q_*((\widetilde{x}_j)_p))),\widehat{g}(\rho(q_*((\widetilde{y}_i)_p))))
    \end{align*}
    By our choice of lifts, the signs of intersections $(t_1,u_1),(t_2,u_2)\in \widetilde{x}_j\sharp \widetilde{y}_i$ are the same if $q(\widetilde{x}_j(t_1))=q(\widetilde{x}_j(t_2))=q(\widetilde{y}_i(u_1))=q(\widetilde{y}_i(u_2))$. It follows that for each of a fixed $j$, we have
    \begin{align*}
     \sum_{i=1}^r \sum_{p\in \widetilde{x}_j \sharp \widetilde{y}_i} (\sign p) \mathsf{B}(\widehat{f}(\rho(q_*((\widetilde{x}_j)_p))),\widehat{g}(\rho(q_*((\widetilde{y}_i)_p))))&=l\sum_{p\in x\sharp y}(\sign p) \kf (\widehat{f}(\rho(x_p)^l),\widehat{g}(\rho(y_p)^m))\\
     &=l\sum_{p\in x\sharp y}(\sign p) \kf (\widehat{f}(\rho(x_p)),\widehat{g}(\rho(y_p))).
    \end{align*}
    Therefore,

    \begin{align}\label{eq:rhs}
     \omega\left(  \sum_{j=1} ^s\mathbb{X}_{f_{\widetilde{x}_j}}, \sum_{i=1}^{r} \mathbb{X}_{g_{\widetilde{y}_i}}\right)&=l\sum_{j=1}^{s}\sum_{p\in x\sharp y}(\sign p) \kf (\widehat{f}(\rho(x_p)),\widehat{g}(\rho(y_p)))\\
     &=|\Delta| \sum_{p\in x \sharp y} (\sign p) \mathsf{B}(\widehat{f}(\rho(x _p)),\widehat{g}(\rho(y _p)))\nonumber.
    \end{align}
    By combining (\ref{eq:lhs}) and (\ref{eq:rhs}), we obtain the desired result.
\end{proof}

For a fixed infinite order $x\in \pi_1 ^{\mathrm{orb}}(O)$ and $c\in \bR$,  the condition $\alpha\circ \lambda(\rho(x))= c$ is not algebraic. However, we still expect that the set $ \{[\rho]\in \Hit_G(O)\,|\, \alpha(\lambda(\rho(x))) = c \}$ is almost negligible.

    \begin{lemma}\label{lem:generic}
    Let $x\in \pi_1 ^{\mathrm{orb}}(O)\setminus\{1\}$ be an infinite order element. Define
    \[
    U_{x,\alpha,c} := \{[\rho]\in \Hit_G(O)\,|\, \alpha(\lambda(\rho(x))) \ne c \}.
    \]
    Then for any $c\in \bR$ and $\alpha\in \mathfrak{a}^*$, $U_{x,\alpha,c}$ is generic provided it is non-empty.
\end{lemma}
\begin{proof}
    Choose a finite manifold covering $\widetilde{O}\to O$. One can identify $\Hit_G(O)$ as a submanifold of $\Hit_G(\widetilde{O})$ fixed by the action of the deck group of $\widetilde{O}\to O$.
    
    It is well-known \cite[Theorem 1.13]{fock} that for a Hitchin representation $\rho$,  $\rho(z)$ is purely loxodromic for any $z \in \pi_1 (\widetilde{O})\setminus\{1\}$.  Since $x$ has infinite order,  $x^k$ is a non-trivial element in $\pi_1(\widetilde{O})$ for some integer $k$. Therefore, $\rho(x^k)$ is purely loxodromic. This implies that $\rho(x)$ is also purely loxodromic. Thus, if $U_{x,\alpha,c}$ is not empty, then it is the complement of the zero locus of the non-constant real analytic function $\rho\mapsto \alpha(\lambda(\rho(x)))$. 
    Since Hitchin components are analytic connected manifold, $U_{x,\alpha,c}$ is open and dense in $\Hit_G(O)$.
\end{proof}
In the above proof, we implicitly showed that $\rho(x)$ is purely loxodromic provided $x$ has infinite order. 

The following result was essentially observed by Goldman in his own proof for Wolpert's cosine formula \cite[Theorem~3.9]{Goldman86}. For readers' convenience, we reformulate it as a lemma. 

\begin{lemma}\label{lem:computesummand}
    Let $\rho:\pi^{\mathrm{orb}}_1(O)\to \PSL_2(\bR)$ be a Fuchsian representation and let $x,y$ be infinite order elements in $\pi_1^{\mathrm{orb}}(O)$. Represent $x,y$ as parameterized closed geodesics in $\mathbb{H}^2/\rho(\pi_1^{\mathrm{orb}}(O))$. For each intersection $p=(t,s)\in x\sharp y$, let $\varphi_p([\rho])$ be the counter-clock angle from $\dot{x}(t)$ to $\dot{y}(s)$ at $x(t)$. Let $f:=\alpha\circ \lambda$ and $g=\beta \circ \lambda$ for some $\alpha,\beta\in \mathfrak{a}^*\setminus\{0\}$. Then 
    \[
    (\operatorname{sign} p) \kf( \widehat{f}(\iota_G\circ \rho(x_p)) , \widehat{g}(\iota_G\circ \rho(y_p)))=\begin{cases}
         \kf( \alpha^\vee , \Ad_{R(\varphi_p([\rho]))} \beta^\vee) & \text{ if }\operatorname{sign} p =1\\
         \kf( \alpha^\vee , -\Ad_{R(\varphi_p([\rho]))} w_0 \beta^\vee) & \text{ if }\operatorname{sign} p = -1
    \end{cases}
    \]
    where 
    \[R(\varphi_p([\rho]))=\iota_G\left( \begin{pmatrix}
       \cos (\varphi_p([\rho])/2) & \sin (\varphi_p([\rho])/2) \\ -\sin (\varphi_p([\rho])/2) & \cos (\varphi_p([\rho])/2) 
    \end{pmatrix}\right).
    \]
\end{lemma}
\begin{proof}
    Take the lifts $\widetilde{x}:\bR\to \mathbb{H}^2$ and $\widetilde{y}:\bR\to \mathbb{H}^2$ of $x$ and $y$ in the upper-half plane model $\mathbb{H}^2$ such that $\widetilde{x}(0)$ and $\widetilde{y}(0)$ descend to the intersection point $x(t)=y(s)$. By taking a suitable $\PSL_2(\bR)$-conjugation, we may assume that $\widetilde{x}$ is the ray from $0$ to $\infty$ and the intersection point $\widetilde{x}(0)=\widetilde{y}(0)$ is situated at $\sqrt{-1}$. Then we know that $\rho(x_p) = \diag(\exp (\ell), \exp (-\ell))$, $\ell >0$, and that the axis of $\rho(y_p)$ passes through $\sqrt{-1}$.  Hence, $\iota_G\circ \rho(x_p)\in \exp\mathfrak{a}^+$. Moreover, conjugating $\rho(y_p)$ by
    \[
    \begin{pmatrix}
       \cos \frac{\varphi_p([\rho])}{2} & -\sin \frac{\varphi_p([\rho])}{2} \\ 
       \sin \frac{\varphi_p([\rho])}{2} & \cos \frac{\varphi_p([\rho])}{2} 
    \end{pmatrix}\] 
    brings $\rho(y_p)$ into $\diag(\exp(\ell'),\exp(-\ell'))$ if $\operatorname{sign}p=1$ and into $\diag(\exp(-\ell'),\exp(\ell'))$ if $\operatorname{sign}p=-1$ for some $\ell'>0$. Therefore, $\Ad_{R(\varphi_p(\rho))}\iota_G\circ \rho(y_p)$ is in $\exp \mathfrak{a}^+$ if $\operatorname{sign} p =1$ and in $\exp -\mathfrak{a}^+$ if $\operatorname{sign} p = -1$. Therefore, the conclusion follows from Lemma~\ref{lem:conjugation} and Lemma~\ref{lem:goldmanfunctioncompute}.
\end{proof}

Now we can show the main theorem. 

\begin{proof}[Proof of Theorem~\ref{thm:genericjordan}]
Let $x\in \pi_1^{\mathrm{orb}}(O)\setminus\{1\}$ be a regular element and let $f:=\alpha\circ \lambda$. As before, we define $f_x:\Hit_G(O)\to \bR$ by $f_x([\rho]) := f(\rho(x))$. Choose a regular oriented closed curve representing $x\in \pi_1^{\mathrm{orb}}(O)$. We denote this oriented closed curve by the same letter $x$.

Let us show (ii) first. Under the assumption $\alpha = -w_0 \alpha$, we will prove that the set $U_{x,\alpha,c}$ in Lemma~\ref{lem:generic} is not empty. 

If $\alpha(\lambda(\rho(x)))\ne c$ for some Fuchsian $\rho$, there is nothing to prove. Hence, we assume that $\alpha(\lambda(\rho(x)))= c$ for all Fuchsian $\rho$. 

Take an oriented essential simple closed curve $y:\bS^1\to O_{\mathrm{reg}}$ that intersects essentially $x$ at least once. This is possible by Lemma~\ref{lem:orbifold}. Let $\rho$ be any Fuchsian representation. Let $\rho_t$ be the 1-parameter family of Fuchsian representations obtained by applying the twisting flow along $y$. We denote by $x_t$ and $y_t$ the (parameterized) geodesic representatives of $x$ and $y$ with respect to the hyperbolic structure $\rho_t$. For each intersection $p\in x_t\sharp y_t$, let $\varphi_p (\rho)$ be the counter-clock angle from $x_t$ to $y_t$ at the intersection $p$.  

The following claim is a folklore but it seems that it has never been explicitly stated. 

\begin{claim}\label{claim:angle}
    $\max_{p\in x_t\sharp y_t} \varphi_p([\rho_t])$ converges to $0$ as $t\to \infty$.
\end{claim}
\begin{proof}
Let $X_t$ be a hyperbolic structure on $S$ with holonomy $\rho_t$. Let $p\in x_t \sharp y_t$ be the intersection point realizing the value $\max_{p\in x_t\sharp y_t} \varphi_p([\rho_t])$.

Kerckhoff \cite[Proposition~3.5]{kerckhoff} showed that, along the path $[\rho_t]$, the angle $\varphi_p([\rho_t])$  at each intersection point $p\in x_t\sharp y_t$ is strictly decreasing. Hence, it suffices to show that for positive integers $k$, $\varphi_{p}([\rho_{k\ell}])\to 0$ as $k\to \infty$ where  $\ell$ is the hyperbolic length of $y$ in $X_0$. Since the time $\ell$ map of the twist flow along $y$ is the same as the action of the Dehn twist $\operatorname{tw}_y$ along $y$, it is enough to show that the angle converges to 0 under the iterate action of the Dehn twist along $y$. 

The angle $\varphi_p([\rho_{k\ell}])$ is the same as the maximal angle between $y_0$ and the geodesic representative of $\operatorname{tw}^{k}_y (x)$ in $X_0$. Since a closed curve and its geodesic representative have the same endpoints in the ideal boundary universal cover $\partial_\infty \mathbb{H}^2$, we investigate the location of the endpoints of the closed curve $\operatorname{tw}^{k}_y (x)$. To find this, we take the inverse image $\widetilde{y}_0$ of the geodesic $y_0$ to the universal cover $\mathbb{H}^2$. A lift of $x_0$ is a geodesic intersecting $\widetilde{y}_0$. The closed curve  $\operatorname{tw}^{k}_y (x)$ is obtained by shifting up by $k\ell$ along a lift of $y$. Hence, as $k\to \infty$, the endpoints of  $\operatorname{tw}^{k}_y (x)$ converge to the endpoints of one of a lift of $y$. See Figure~\ref{fig:dehn}.\end{proof}
\begin{figure}[hbt]
    \centering
    \includegraphics[width=0.9\linewidth]{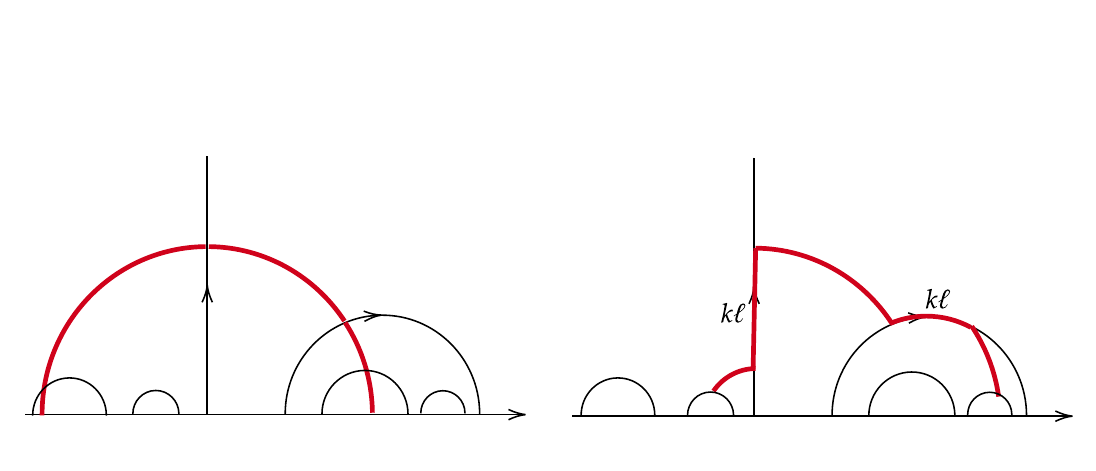}
    \caption{The effect of the Dehn twist. Black geodesics are lifts of $y$ and red geodesic is a lift of $x$ in $\mathbb{H}^2$. The red curve in the right hand side is a lift of $\operatorname{tw}^{k}_y(x)$.}
    \label{fig:dehn}
\end{figure}

Due to Claim~\ref{claim:angle}, $R(\varphi_p([\rho_t]))$ converges to the identity as $t\to \infty$. Thus we can take a sufficiently large $t_0$ such that 
\[
    \kf( \alpha^\vee, \Ad_{R(\varphi_p([\rho_{t_0}]))} \alpha^\vee )>0
\]
for all $p\in x_{t_0}\sharp y_{t_0}$. Let $\rho' := \rho_{t_0}$ and let $[\rho'_s]$  be the Goldman flow associated to the vector field $\mathbb{X}_{f_y}$ with $[\rho'_0] = [\rho']$. Since $\alpha= -w_0 \alpha$,  Lemma~\ref{lem:generalderivative} and  Lemma~\ref{lem:computesummand} tell us that
\[
\left.\frac{\rd}{\rd s} \right|_{s=0} f_x ([\rho'_s]) = \sum_{p\in x_{t_0}\sharp y_{t_0}} \kf( \alpha^\vee , \Ad_{R(\varphi_p([\rho']))} \alpha^\vee).
\]
By our choice of $\rho'$, we have
\[
\left.\frac{\rd}{\rd s} \right|_{s=0} f_x ( [\rho' _s]) >0.
\]
Therefore, one can find a sufficiently small $s>0$ so that $f_x([\rho_s'])=f(\rho'_s (x))>f(\rho'(x))= c$. This shows  that $U_{x,\alpha, c} \ne \emptyset$. 

Since
    \begin{multline*}
    \{[\rho]\in \Hit_G(O) \,|\, \alpha\circ \lambda (\rho(x)) \ne c \text{ for all regular  }x\notin \ker(\pi_1^{\mathrm{orb}}(O)\to \pi_1^{\mathrm{orb}}(O)^{\mathrm{ab}})\}\\
    =\bigcap_{\substack{w\notin \ker(\pi_1^{\mathrm{orb}}(O)\to \pi_1^{\mathrm{orb}}(O)^{\mathrm{ab}})\\ w\text{ regular}}} U_{w,\alpha, c},
    \end{multline*}
    the theorem follows from Lemma~\ref{lem:generic}.

Now we show (i). Again, we want to prove that $U_{x,\alpha,c}$ is non-empty. Hence assume that all Fuchsian representations are not in the set $U_{x,\alpha,c}$. 

Since $x$ is not in the kernel of the abelianization, $x$ is non-zero even in the homology $H_1(O_{\mathrm{reg}})$. Since $x$ is regular, we can find an oriented regular simple closed curve $y$ in $O_{\mathrm{reg}}$ such that the algebraic intersection number $\ri(x,y)$ is positive. 

Let $\alpha_+:=\frac{1}{2}(\alpha+ w_0 \alpha)$ and $\alpha_- := \frac{1}{2}(\alpha - w_0 \alpha)$. Observe that
\[
\kf (\alpha_+^\vee, \alpha_-^\vee)  = 0.
\]
We may assume that $\alpha_+\ne 0$. Indeed, the $\alpha_+=0$ case was already handled in item (ii) above. Choose a positive number $\epsilon$ such that
\[
\epsilon < \frac{1}{8 |x\sharp y |}\kf(\alpha_+^\vee, \alpha_+ ^\vee).
\]
Let $\rho_t$ be the twisting flow along $y$ emanating from a Fuchsian representation $\rho$. As in the above, we may choose a sufficiently large $t_0$ so that 
\begin{align*}
\left|\kf(\alpha_+ ^\vee, \Ad_{R(\varphi_p([\rho_{t_0}]))} \alpha_+^\vee) - \kf (\alpha_+ ^\vee, \alpha_+ ^\vee) \right| &<\epsilon\\
\left|\kf(\alpha_+ ^\vee, \Ad_{R(\varphi_p([\rho_{t_0}]))} \alpha_-^\vee) \right| &<\epsilon
\end{align*}
for all $p\in x_{t_0}\sharp y_{t_0}$. 

Let $(x_{t_0}\sharp y_{t_0})_+$ be the set of positive intersections and let $(x_{t_0}\sharp y_{t_0})_-$ be the set of negative intersections. Let $\rho'_s$ be the Goldman flow along $y$ with $\rho'_0 =\rho':= \rho_{t_0}$ associated to the invariant function $\alpha_+\circ \lambda$. Since $\alpha = \alpha_++\alpha_-$,  we can write
\begin{multline*}
    \left. \frac{\rd}{\rd s}\right|_{s=0}f(\rho'_s(x))= \sum_{p\in (x_{t_0}\sharp y_{t_0})_+ } \kf (\alpha_+ ^\vee, \Ad_{R(\varphi_p([\rho']))}\alpha_+ ^\vee )-\sum_{p\in (x_{t_0}\sharp y_{t_0})_-}\kf (\alpha_+ ^\vee,\Ad_{R(\varphi_p([\rho']))} \alpha_+ ^\vee )\\
    +\sum_{p\in x_{t_0}\sharp y_{t_0}} \kf (\alpha_- ^\vee, \Ad_{R(\varphi_p([\rho']))}\alpha_+ ^\vee).
\end{multline*}
From the construction, we know that
\begin{align*}
\left. \frac{\rd}{\rd s}\right|_{s=0}f(\rho'_s(x))&>  \ri(x,y) \left( \kf (\alpha_+ ^\vee, \alpha_+ ^\vee ) -\epsilon\right) -2 \epsilon|(x\sharp y)_-| - \epsilon |x\sharp y| \\
&\ge \ri(x,y)\kf (\alpha_+ ^\vee, \alpha_+ ^\vee ) -\epsilon \ri(x,y)  -3\epsilon|x\sharp y|\\
&\ge \ri(x,y)\kf (\alpha_+ ^\vee, \alpha_+ ^\vee ) - 4 \epsilon |x\sharp y|\\
&>\left( \ri(x,y) -\frac{1}{2}\right)\kf (\alpha_+ ^\vee, \alpha_+ ^\vee )\\
&> 0.
\end{align*}
Therefore, there is a sufficiently small $s_0$ such that $\rho'_{s_0}\in U_{x,\alpha,c}$. To conclude the proof, we use the same argument as in the proof for (ii) above. 
\end{proof}

The idea of the proof for Theorem~\ref{thm:genericjordan}  carries over to the proof for Theorem~\ref{thm:main2}.

\begin{proof}[Proof of Theorem~\ref{thm:main2}] Following \cite{long}, let $\operatorname{Bad}(x,y)$ be the set of $G$-Hitchin representations $\rho$ such that $\rho(x)$ and $\rho(y)$ do not generate a Zariski dense subgroup of $G$. Since $\operatorname{Bad}(x,y)$ is an analytic subvariety of $\overline{\Hit}_G(O)$ \cite[Theorem~4.1]{tao}, it suffices to show that $\operatorname{Bad}(x,y)$ is proper for every non-commuting pair $\{x,y\}$. 

By mean of contradiction, suppose that there are non-commuting $x,y\in \pi_1 ^{\mathrm{orb}}(O)$ such that $\operatorname{Bad}(x,y)=\overline{\Hit}_G(O)$. 

Choose a finite index torsion-free subgroup $\Gamma$ of $\langle x,y\rangle$ and take a covering $\widetilde{O}$ of $O$ corresponding to $\Gamma$. Let $\eta$ be in the Fuchsian locus of $\overline{\Hit}_G(O)$. Because $\eta|\Gamma=\eta|\pi_1(\widetilde{O})$ is the holonomy of the lifted hyperbolic structure on $\widetilde{O}$, we know that $\eta|\Gamma$ is a positive representation in the sense of Fock and Goncharov \cite{fock}. Since positivity is an open condition, one can find a small open set $V\subset \Hit_G(O)$ around $[\eta]$ such that for any $[\rho]\in V$, $[\rho|\Gamma]$ is positive. Let $\overline{V}$ be the preimage of $V$ under the projection $\overline{\Hit}_G(O) \to \Hit_G(O)$.

For $\rho\in \overline{V}$, let $\rho_\Gamma:=\rho|\Gamma:\Gamma \to G$ and let $H_{\rho_\Gamma}$ be the Zariski closure of $\rho(\Gamma)$. By assumption, we know that $H_{\rho_\Gamma}\ne G$ for all $\rho\in \overline{V}$. 

Consider the semisimplification $\rho_\Gamma ^{ss}:\Gamma \to H_{\rho_\Gamma}^{ss}$ of the representation  $\rho_\Gamma$. Note that $H_{\rho_\Gamma} ^{ss}$ is a proper semi-simple Lie subgroup of $G$. It is known \cite[Lemma 2.40]{gueritaud} that $\lambda (\rho_\Gamma ^{ss}(z)) = \lambda(\rho_\Gamma(z))=\lambda(\rho(z))$ for all $z\in \Gamma$.  

By Sambarino \cite[Theorem~B]{sambarino}, there are only finitely many semi-simple subalgebras $\mathfrak{h}_1, \mathfrak{h}_2, \cdots, \mathfrak{h}_k$ of $\mathfrak{g}$ whose real ranks are strictly smaller than the real rank of $G$ such that, apart from the Zariski dense case, the Lie algebra of $H_{\rho_\Gamma} ^{ss}$ is conjugate to one of $\mathfrak{h}_i$. Therefore, one can find finitely many elements $\alpha_1,  \alpha_2,\cdots,\alpha_m$ in $\mathfrak{a}^*\setminus\{0\}$ such that for any $\rho \in \overline{V}$ and $z\in \Gamma$, we have $\alpha_i (\lambda(\rho_\Gamma^{ss} (z)))=\alpha_i (\lambda(\rho_\Gamma(z)))=\alpha_i (\lambda(\rho(z)))= 0$ for some $i$. This means that the set $\bigcap_{i=1} ^m U_{z,\alpha_i, 0}$ avoids the non-empty open set $V$. 

On the other hand, because $O$ has at most one even order cone point, $\langle x,y\rangle$ contains a regular $z\in \pi_1^{\mathrm{orb}}(O)$ in $\Gamma$. By the proof of Theorem~\ref{thm:genericjordan}~(ii), we know that $\bigcap_{i=1} ^m U_{z,\alpha_i, 0}$ is an open dense subspace of $\Hit_G(O)$. This leads us to a contradiction. 
\end{proof}

\section{Further Discussions and Questions}\label{sec:discussion}

In this section, we focus on the $G=\PSL_3(\bR)$ case and present explicit formulas computing the first variation of eigenvalues along Fuchsian locus. In what follows we abbreviate $\Hit_3(S) := \Hit_{\PSL_3(\bR)}(S)$. We also use the trace as our non-degenerate invariant symmetric bilinear form $\kf$.

Let $c$ be an oriented essential simple closed curve $c$.  
In $\Hit_3(S)$, we write $\ell:=\lambda_1-\lambda_3$.  As before, $\lambda_i$ denotes the logarithm of the $i$th largest eigenvalue. The quantity $\ell_c([\rho])=\lambda_1(\rho(c))-\lambda_3(\rho(c))$, $[\rho]\in \Hit_3(S)$ records the Hilbert length of $c$ with respect to the convex real projective geometry corresponding to $[\rho]$. 

There are two types of fundamental Hamiltonian flows.  One is the twisting flow $\mathcal{T}^c_t$ along $c$ which is the same as the classical twisting flow on the Teichm\"ulle space. The other flow is so-called the bulging flow $\mathcal{B}^c _t$ along $c$. The bulging flow allows us to move away from the Fuchsian locus.  It is known that $\mathcal{T}^c_t$ and $\mathcal{B}^c_t$ are Hamiltonian flows corresponding to the vector fields $\mathbb{X}_{\ell_c}$ and $\mathbb{X}_{(\lambda_2)_c}$ respectively.

Recall that one possible embedding $\iota_{\PSL_3(\bR)}:\PSL_2(\bR)\to \PSL_3(\bR)$ is given by
\begin{equation}\label{eq:embedding}
\begin{pmatrix}
    a & b \\ c & d 
\end{pmatrix}\mapsto \begin{pmatrix}
    a^2 & ab & b^2 \\ 2ac & ad+bc & 2bd\\ c^2 & cd & d^2
\end{pmatrix}.
\end{equation}
A straightforward computation using (\ref{eq:embedding}) and  Lemma~\ref{lem:computesummand} yields the following. 
    
    \begin{corollary}\label{cor:variation}
        Let $x$ be an oriented essential closed curve  and let $y$ be an oriented essential simple closed curve.  Let $[\rho]\in \Hit_3(S)$ be a Fuchsian representation. Then we have
        \begin{align}
            \left.\frac{\rd}{\rd t}\right|_{t=0} \ell_x (\mathcal{T}^y _t ([\rho])) &= \sum_{p \in x\sharp y} 2 \cos\varphi_p ([\rho])\label{eq:twisting}\\
            \left.\frac{\rd}{\rd t}\right|_{t=0}(\lambda_2)_x (\mathcal{B}_t ^y ([\rho])) &=\sum_{p\in x\sharp y}  \frac{3}{2} (\operatorname{sign} p) (1+3 \cos  2\varphi_p([\rho]))\label{eq:bulging}\\
            \left.\frac{\rd}{\rd t}\right|_{t=0}\ell_x (\mathcal{B}_t ^y ([\rho])) &=\left.\frac{\rd}{\rd t}\right|_{t=0}(\lambda_2)_x (\mathcal{T}_t ^y ([\rho]))=0 \label{eq:noeffect},
        \end{align}        
        where $\varphi_p([\rho])$ is the counter-clock angle at $p$ from geodesic representatives of $x$ to $y$ measured in the hyperbolic structure corresponding to $[\rho]$.
    \end{corollary}
    
    In Corollary~\ref{cor:variation}, (\ref{eq:twisting}) is nothing but Wolpert's cosine formula, also featured in Kerckhoff's work \cite{wolpert, kerckhoff}. When $x$ is simple, (\ref{eq:noeffect}) was manifested in \cite[Lemma 5.1.1]{wentworth}. 

    Given a closed orientable surface $S$ with genus $>1$, there are infinitely many pairs of non-conjugate elements $x,y\in \pi_1(S)$ such that $\operatorname{length}_X(x)= \operatorname{length}_X(y)$ for all hyperbolic structures $X$ on $S$ \cite{randol}. By (\ref{eq:twisting}) we have an immediate but rather surprising consequence:
    \begin{corollary}\label{cor:sameangle}
        There are infinitely many pairs $x,y\in \pi_1(S)$ of non-conjugating elements such that
        \[
        \sum_{p\in c\sharp x }\cos \varphi_p ([\rho]) = \sum_{q\in c\sharp y} \cos \varphi_q([\rho])
        \]
        for all essential simple closed curves $c$ and all $[\rho]$ in the Teichm\"uller space. 
    \end{corollary}

    Leininger \cite[Theorem 1.4]{leininger} proved that for such a pair $x,y$ in Corollary~\ref{cor:sameangle}, the following also holds:
    \begin{theorem}[\cite{leininger}]\label{thm:trace}
        There are infinitely many pairs of non-conjugate elements $x,y\in \pi_1(S)$ such that $\operatorname{Tr}(\rho(x))= \operatorname{Tr}(\rho(y))$ for all quasi-Fuchsian $\rho:\pi_1(S)\to \PSL_2(\mathbb{C})$.
    \end{theorem}
Motivated from Theorem~\ref{thm:trace}, we may ask the following question question regarding Hitchin representations.
    \begin{question}\label{qu:trace}
        Is there a pair of non-conjugate elements $x,y\in \pi_1(S)$ such that $\operatorname{Tr}(\rho(x))= \operatorname{Tr}(\rho(y))$ for all $[\rho]\in \Hit_3(S)$? What if we replace $\operatorname{Tr}$ with $\lambda_2$?
    \end{question}

In view of Corollary~\ref{cor:variation}, some questions regarding generic properties of $\Hit_3(S)$ boil down to hyperbolic geometry problems on $S$. For instance, the following question is very closely related to  Question~\ref{qu:trace} via (\ref{eq:bulging}). 
    \begin{question}\label{qu:angle}
        Let $x,y\in \pi_1(S)$ be a non-conjugate pair. Is there a hyperbolic structure $[\rho]$ and an oriented essential simple closed curve $c$ such that 
                \[
        \sum_{p\in c\sharp x }(\operatorname{sign} p) \cos 2\varphi_p ([\rho]) \ne \sum_{q\in c\sharp y} (\operatorname{sign} q)  \cos 2\varphi_q([\rho])
        \]
        holds?
    \end{question}
At first glance, the answer to Question~\ref{qu:angle} should be easy yes. However, given Corollary~\ref{cor:sameangle},  resolving Question~\ref{qu:angle} might not be as completely obvious as it looks. 

\appendix

\section{Analyticity of the Jordan Projection}\label{app:open}
In this appendix we prove Lemma~\ref{lem:openness}. We restate the claim first.
\begin{lemma} Let $G$ be a split real form of a complex simple adjoint Lie group. Let $\lambda:G\to \overline{\mathfrak{a}^+}$ be the Jordan projection and let $U$ be the set of purely loxodromic elements. Then,
\begin{itemize}
    \item[(i)] $U$ is open,
    \item[(ii)] $\lambda$ is real analytic on $U$.
\end{itemize}
\end{lemma}
\begin{proof}
    (i) Consider the adjoint representation $\Ad:G\to \operatorname{SL}(\mathfrak{g})$. The Jordan projection is equivariant in the sense that $\lambda(\Ad_x) = \operatorname{ad}_{\lambda(x)}$ for all $x\in G$. The image $C$ of $\partial\overline{\mathfrak{a}^+}$ under the representation $\operatorname{ad}$ is a linear subspace of
    \[
    \{\diag (0,0,\cdots, 0, \beta_ 1, \cdots, \beta_d)\,|\, \beta_i\ge 0 \text{ and }\beta_1\beta_2 \cdots \beta_d = 0\}
    \]
    where $d= \dim\mathfrak{g}-\dim\mathfrak{a}$. In particular, $C$ is a closed set. Therefore we only need to show that 
    \[
    D:=\{x\in \operatorname{SL}(\mathfrak{g})\,|\, \lambda(x) \in C\}
    \]
    is closed. This can be proven by using the fact that the roots of a polynomial depend continuously on the coefficients. In particular, if $g_i$ is a given sequence converging to $g$ with $\lambda(g_i)\in C$, then because the eigenvalues of $g_i$ converges to that of $g$, we have $\lambda(g_i) \to \lambda(g)$ as described in Example~\ref{ex:psl}. Since $C$ is closed, we know that $\lambda(g)\in C$. Thus, $D$ is closed. Since $U$ is the complement of the closed set $\Ad^{-1}(D)$, $U$ is open. 

    (ii) Let $A^+ = \exp \mathfrak{a}^+$. It is enough to show that $\lambda$ is analytic at a neighborhood of each point of $A^+$. Consider the analytic map $F:A^+\times G \to U$ defined by
    \[
    (a,x) \mapsto xax^{-1}.
    \]
    Let $a_0:=\exp A_0\in A^+\subset U$ be arbitrary. Take exponential charts containing $(a_0,1)\in A^+ \times G$ and  $a_0\in U$. If we express $F$ in these charts, we would get
    \[
    F: (A,X) \mapsto e^{\ad_X}A
    \]
    where $A\in \mathfrak{a}^+$, $X\in \mathfrak{g}$ and $e^{\ad_X}=\sum_{k\ge 0}\frac{1}{k!}(\ad_X)^k$. Recall that we have the restricted root space decomposition $\mathfrak{g}=\mathfrak{a} \oplus \bigoplus_{\beta\in \Pi}\mathfrak{g}_\beta$. Let $\Delta=\{H_1,\cdots, H_r\}$ be elements  in $\mathfrak{a}$ dual to simple roots. Write $\Pi=\{\beta_1, \cdots, \beta_d\}$ and choose a non-zero element $E_i$ from each $\mathfrak{g}_{\beta_i}$. Then $\Delta\cup\{E_i\}_{i=1,2,\cdots,d}$ forms a basis for $\mathfrak{g}$. With respect to this basis, we write $A$ and $X$ as coordinates 
    \begin{align*}
        A&= A_0+\sum_{i=1} ^r y_i H_i,\\
        X&=\sum_{i=1}^r x_i H_i+\sum_{i=1}^d z_i E_i,\\
        F(A,X)&=\sum_{i=1}^ r f_i(A,X)H_i + \sum_{i=1}^d g_i(A,X) E_i.
    \end{align*}

    Now we compute the Jacobian of $F$ at $(a_0,1)$:
    \[
    \frac{\partial f_i}{\partial y_j} = \delta_{ij},\qquad \frac{\partial g_i}{\partial y_j} = 0,
    \]
    
    \[
    \frac{\partial f_i}{\partial x_j} = 0,\qquad \frac{\partial g_i}{\partial x_j} = 0, 
    \]
    and
    \[
    \frac{\partial f_i}{\partial z_j} = 0,\qquad \frac{\partial g_i}{\partial z_j} = -\delta_{ij}\beta_i(A_0),
    \]
where $\delta_{ij}$ is the Kronecker delta. Since $A_0\in \mathfrak{a}^+$, we have $\beta_i(A_0)\ne 0$ for all $i$. Therefore, $F$ has the full rank at $(a_0,1)$ and by the analytic implicit function theorem, we can find an analytic local section $s:O\to A^+\times G$ of $F$ in a neighborhood $O$  of $a_0$. The logarithm of the $A^+$-factor of $s(a)$ is precisely the Jordan projection of $a\in O$. Thus $\lambda$ is analytic.
\end{proof}

    \bibliographystyle{alpha}
    \bibliography{ref.bib}

\end{document}